\documentclass[12pt,english]{amsart}
\usepackage{ae,aecompl}
\usepackage[T1]{fontenc}
\usepackage[latin9]{inputenc}
\usepackage[a4paper]{geometry}
\usepackage[active]{srcltx}
\usepackage{verbatim}
\usepackage{mathtools}
\usepackage{amsthm}
\usepackage{amssymb}
\usepackage{stmaryrd}
\usepackage{graphicx}
\usepackage[all]{xy}
\usepackage{babel}

\makeatletter

\providecommand{\tabularnewline}{\\}

\theoremstyle{plain}
\newtheorem{thm}{\protect\theoremname}[section]
  \theoremstyle{plain}
  \newtheorem{prop}[thm]{\protect\propositionname}
  \theoremstyle{definition}
  \newtheorem{defn}[thm]{\protect\definitionname}
  \theoremstyle{remark}
  \newtheorem{rem}[thm]{\protect\remarkname}
  \theoremstyle{definition}
  \newtheorem{example}[thm]{\protect\examplename}
  \theoremstyle{plain}
  \newtheorem{lem}[thm]{\protect\lemmaname}
  \theoremstyle{plain}
  \newtheorem{cor}[thm]{\protect\corollaryname}
  \theoremstyle{plain}
  \newtheorem{conjecture}[thm]{\protect\conjecturename}

\def\quot{/\!\!/}
\def\sym{\mathsf{Sym}}

\def\odd{\scriptsize \mathrm{odd}}
\def\even{\scriptsize \mathrm{even}}
\def\hom{\mathsf{Hom}}

\usepackage{babel}

\title[Hodge-Deligne polynomials of character varieties]
{Hodge-Deligne polynomials of \\ character varieties of free abelian groups}

\author[C. Florentino]{Carlos Florentino}

\address{Departamento de Matem\'{a}tica, Faculdade de Ci\^{e}ncias, Univ. de Lisboa,  Edf. C6, Campo Grande 1749-016 Lisboa, Portugal}

\email{caflorentino@fc.ul.pt}

\author[J. Silva]{Jaime Silva}

\address{Departamento de Matem\'{a}tica, ISEL - Instituto Superior de Engenharia de Lisboa, Rua Conselheiro Emídio Navarro, 1, 1959-007 Lisboa, Portugal}

\email{jaime.a.m.silva@gmail.com}

\thanks{The authors acknowledge support from the projects PTDC/MAT-PUR/30234/2017 and EXCL/MAT-GEO/0222/2012, FCT, Portugal, and ``RNMS: GEometric structures And Representation varieties'' (the GEAR Network), U. S. National Science Foundation. J. Silva was supported by the FCT grant SFRH/BD/84967/2012.}

\@namedef{subjclassname@2020}{\textup{2020} Mathematics Subject Classification}

\makeatother

  \providecommand{\conjecturename}{Conjecture}
  \providecommand{\corollaryname}{Corollary}
  \providecommand{\definitionname}{Definition}
  \providecommand{\examplename}{Example}
  \providecommand{\lemmaname}{Lemma}
  \providecommand{\propositionname}{Proposition}
  \providecommand{\remarkname}{Remark}
\providecommand{\theoremname}{Theorem}

\begin{document}
\begin{abstract}
Let $F$ be a finite group and $X$ be a complex quasi-projective
$F$-variety. For $r\in\mathbb{N}$, we consider the mixed Hodge-Deligne
polynomials of quotients $X^{r}/F$ where $F$ acts diagonally, and
compute them for certain classes of varieties $X$ with simple mixed
Hodge structures. A particularly interesting case is when $X$ is
the maximal torus of an affine reductive group $G$, and $F$ is its
Weyl group. As an application, we obtain explicit formulae for the
Hodge-Deligne and $E$-polynomials of (the distinguished component
of) $G$-character varieties of free abelian groups. In the cases
$G=GL(n,\mathbb{C})$ and $SL(n,\mathbb{C})$ we get even more concrete
expressions for these polynomials, using the combinatorics of partitions.
\end{abstract}

\maketitle

\section{Introduction}

\thispagestyle{empty}

The study of the geometry, topology and arithmetic of character varieties
is an important topic of contemporary research. Given a reductive
complex algebraic group $G$, and a finitely presented group $\Gamma$,
the $G$-character variety of $\Gamma$ is the (affine) geometric
invariant theory (GIT) quotient 
\[
\mathcal{M}_{\Gamma}G:=\hom(\Gamma,G)\quot G.
\]

When the group $\Gamma$ is the fundamental group of a Riemann surface
(or more generally, a Kähler group) these spaces are homeomorphic
to moduli spaces of $G$-Higgs bundles via the non-abelian Hodge correspondence
(see, for example, \cite{Hi,Si}), and have found interesting connections
to important problems in Mathematical Physics in the context of mirror
symmetry and the geometric Langlands correspondence.

Recently, some interesting formulas were obtained by Hausel, Letellier
and Rodriguez-Villegas for the so-called $E$-polynomial of smooth
$GL(n,\mathbb{C})$-character varieties of surface groups, by applying
arithmetic harmonic analysis to their $\mathbb{Z}$-models, and proving
these are polynomial count \cite{HRV,HLRV}. By computing indecomposable
bundles on algebraic curves over finite fields, Schiffmann determined
the Poincaré polynomial of the moduli spaces of stable Higgs bundles,
hence of the corresponding $GL(n,\mathbb{C})$-character varieties
of surface groups \cite{Sc}. Other methods based on point counting were employed by Mereb \cite{Me}
(the $SL(n,\mathbb{C})$ case) and Baraglia-Hekmati \cite{BH} (the
singular, small $n$ case). 

Moreover, geometric tools were developed by Lawton, Logares, Muñoz
and Newstead to calculate the $E$-polynomials using stratifications
of character varieties (over $\mathbb{C}$) of surface groups, exploring
directly the additivity of these polynomials \cite{LMN,LM}. This
lead to the development of a Topological Quantum Field Theory for
character varieties by González-Prieto et al \cite{GPLM,GP}.

In the present article, we deal instead with $G$-character varieties
of \emph{free abelian groups}, and with the determination of their
\emph{mixed Hodge structures} (MHS)  for a general complex reductive
$G$. In particular, we explicitly compute the mixed Hodge polynomials
of these varieties. The mixed Hodge polynomial $\mu_{X}$ is a 3 variable
polynomial $\mu_{X}=\mu_{X}(t,u,v)$ defined for any (complex) quasi-projective
variety $X$ and encodes all numerical information about the MHS on
the cohomology of $X$, generalizing both the Poincaré and the $E$-polynomials.

To present our main results, denote the $G$-character variety of
the free abelian group $\Gamma\cong\mathbb{Z}^{r}$, $r\in\mathbb{N}$,
by: 
\[
\mathcal{M}_{r}G:=\mathcal{M}_{\mathbb{Z}^{r}}G=\hom(\mathbb{Z}^{r},G)\quot G,
\]
where $\quot$ stands for the (affine) \emph{geometric invariant theory}
(GIT) quotient (see e.g. \cite{MFK,Mu}) for the natural $G$-action,
by conjugation, on the space of representations $\hom(\mathbb{Z}^{r},G)$.
This later space consists of pairwise commuting $r$-tuples of elements
of $G$, and is of relevance in Mathematical Physics, namely in the
context of supersymmetric Yang-Mills theory \cite{KS}. When $r$
is even, $\mathbb{Z}^{r}$ is also a Kähler group (the fundamental
group of a Kähler manifold) and the smooth locus of $\mathcal{M}_{\mathbb{Z}^{2m}}(G)$
is diffeomorphic to a certain moduli space of $G$-Higgs bundles over
a $m$-dimensional abelian variety (see, for instance \cite{BF}).

The topology and geometry of character varieties of free abelian groups
has been studied by Florentino-Lawton, Sikora, and Ramras-Stafa, among
others (see, eg, \cite{FL,Si2,RS,St}). It is known that the affine
algebraic variety $\mathcal{M}_{r}G$ is not in general irreducible,
but the irreducible component of the trivial $\mathbb{Z}^{r}$-representation,
denoted $\mathcal{M}_{r}^{0}G$, has a normalization $\mathcal{M}_{r}^{\star}G$
isomorphic to $T^{r}/W$, (\cite[Thm 2.1]{Si2}) where $T\subset G$
is a maximal torus and $W$ is the Weyl group, acting diagonally on
$T^{r}$ (hence also on its cohomology). Thus, the varieties $\mathcal{M}_{r}^{\star}G$
are singular orbifolds of dimension $r\dim T$ with a special kind
of mixed Hodge structures, called \emph{balanced }or of Hodge-Tate
type and they satisfy the analogue of Poincaré duality for MHS. When
$r=2$, Thaddeus proved that $\mathcal{M}_{2}^{\star}G$ are of crucial
importance in mirror symmetry and Langlands duality, and computed
their \emph{orbifold} $E$-polynomials \cite{Th}. Here, we obtain
the following explicit formula for \emph{mixed Hodge polynomials}
of $\mathcal{M}_{r}^{\star}G$.
\begin{thm}
\label{thm:mHDP}Let $r\geq1$, $G$ be a complex reductive group
with maximal torus $T$ and Weyl group $W$. Then, 
\begin{equation}
\mu_{\mathcal{M}_{r}^{\star}G}(t,u,v)=\frac{1}{|W|}\sum_{g\in W}\det\left(I+tuv\,A_{g}\right)^{r},\label{eq:main}
\end{equation}
where $A_{g}$ is the automorphism induced on $H^{1}(T,\mathbb{C})$
by $g\in W$, and $I$ is the identity automorphism. 
\end{thm}

One consequence of this result is a formula for the (compactly supported)
$E$-polynomial of the irreducible component $\mathcal{M}_{r}^{0}G\subset\mathcal{M}_{r}G$,
for every such $G$ (Theorem \ref{thm:E-poly}).

Our approach to Theorem \ref{thm:mHDP} is based on working with \emph{equivariant}
mixed Hodge structures and their corresponding \emph{equivariant polynomials},
defined for varieties with an action of a finite group, and focusing
on certain classes of balanced varieties. In particular, we generalize
to the context of the equivariant $E$-polynomial, some of the techniques
introduced in \cite{DL} for dealing with equivariant weight polynomials.

For the groups $G=GL(n,\mathbb{C})$ and $SL(n,\mathbb{C})$, we have
that $\mathcal{M}_{r}G$ is an irreducible normal variety, and the
formula in Theorem \ref{thm:mHDP} can be made even more concrete,
in terms of partitions of $n$, and allows explicit computations of
the Hodge-Deligne, $E$- and Poincaré polynomials of the corresponding
character varieties $\mathcal{M}_{r}G=\mathcal{M}_{r}^{\star}G$.
We state the main results below in the compactly supported version,
the one which is relevant in arithmetic geometry (see \cite{HRV},
Appendix). 

Let $\mathcal{P}_{n}$ denote the set of partitions of $n\in\mathbb{N}$.
By $\underline{n}=[1^{a_{1}}2^{a_{2}}\cdots n^{a_{n}}]\in\mathcal{P}_{n}$
we denote the partition of $n$ with $a_{j}\geq0$ parts of size $j=1,\cdots,n$,
so that $n=\sum_{j}j\,a_{j}$. 
\begin{thm}
\label{thm:E-SLn} Let $G=SL(n,\mathbb{C})$, and $r\geq1$. The compactly
supported $E$-polynomial of $\mathcal{M}_{r}G$ is
\[
E_{\mathcal{M}_{r}G}^{c}(x)=\frac{1}{(x-1)^{r}}\sum_{\underline{n}\in\mathcal{P}_{n}}\prod_{j=1}^{n}\frac{(x^{j}-1)^{a_{j}r}}{a_{j}!\,j^{a_{j}}},
\]
where $\underline{n}=[1^{a_{1}}2^{a_{2}}\cdots n^{a_{n}}]\in\mathcal{P}_{n}$. 
\end{thm}

Theorem \ref{thm:E-SLn} generalizes, to every $r,n\geq1$, some formulas
recently obtained in \cite{BH,LM} (the cases $n=2$ and $n=3$) by
different methods, which are only tractable for low values of $n$:
the approach in \cite{LMN,LM} uses stratifications and fibrations
to compute $E$-polynomials of character varieties of free groups,
respectively, surface groups; the computations in \cite{BH} apply
representation theory of finite groups, and point counting of varieties
over finite fields.

By substituting $x=1$ in $E_{\mathcal{M}_{r}G}^{c}$, we obtain the
Euler characteristics of these moduli spaces. Moreover, by showing
that $\mathcal{M}_{r}G$ have very special mixed Hodge structures
(MHS) (that we call \emph{round}, see Definition \ref{def:Round})
Theorems \ref{thm:E-SLn} and immediately provide explicit formulas
for their mixed and Poincaré polynomials (Theorem \ref{thm:mu-GLn}).

The $GL(n,\mathbb{C})$ case is particularly symmetric, as the generating
function of mixed Hodge polynomials gives precisely the formula of
J. Cheah \cite{Ch} for the mixed Hodge numbers of symmetric products.
On the other hand, by examining the action of $W$ on the cohomology
of a maximal torus, our methods allow for the computation of $\mu_{\mathcal{M}_{r}G}$
for all the classical complex semisimple groups $G$. These will be
addressed in upcoming work.

We now outline the contents of the article. In section 2, we review
necessary background on MHS, quasi-projective varieties, etc, and
define the relevant polynomials, providing examples, and focusing
on balanced varieties. In section 3, we study properties of special
MHS, related with notions defined in \cite{DL}, and paying special
attention to round varieties, for which the knowledge of either the
Poincaré polynomial or the $E$-polynomial allows the determination
of $\mu$. Section 4 is devoted to equivariant MHS, character formulae
and the cohomology of finite quotients. Finally, in Section 5 we prove
our main Theorem and provide explicit calculations of Hodge-Deligne
and $E$-polynomials (and Euler characteristics) of character varieties
of $\mathbb{Z}^{r}$, in particular for $GL(n,\mathbb{C})$ and $SL(n,\mathbb{C})$;
in the $GL(n,\mathbb{C})$ case the computations are related to MHS
on symmetric products, thereby obtaining a curious combinatorial identity.
In the Appendix, we present a proof, based on \cite{DL}, of the equivariant
version of a theorem in \cite{LMN,LM} on the multiplicative property
of the $E$-polynomial for fibrations.

A preliminary version of the main results has been announced in \cite{FNSZ}.

\textbf{Acknowledgements.} We would like to thank many interesting
and useful conversations with colleagues on topics around mixed Hodge
structures, especially T. Baier, I. Biswas, P. Boavida, G. Granja,
S. Lawton, M. Logares, V. Muñoz and A. Oliveira; and the anonymous
referee for the careful and thorough reading, and suggestions leading
to several improvements. Thanks are also due to the Simons Center
for Geometry and Physics, and the hosts of a 2016 workshop on Higgs
bundles, where some of the ideas herein started to take shape. The
first author dedicates this article to E. Bifet, for his mathematical
enthusiasm on themes close to this one, which had a long lasting influence.

\section{Preliminaries on Character Varieties and on Mixed Hodge Structures}

We start by recalling the relevant definitions and properties of character
varieties and of mixed Hodge structures (MHS) on quasi-projective
varieties, which serves to fix terminology and notation.

\subsection{Character varieties}

Given a finitely generated group $\Gamma$ and a complex affine reductive
group $G$, the $G$-character variety of $\Gamma$ is defined to
be the (affine) Geometric Invariant Theory (GIT) quotient (see \cite{MFK,Mu};
\cite{Sw} for topological aspects): 
\[
\mathcal{M}_{\Gamma}(G)=\hom(\Gamma,G)\quot G.
\]
Note that $\hom(\Gamma,G)$, the space of homomorphisms $\rho:\Gamma\to G$,
is an affine variety, as $\Gamma$ is defined by algebraic relations,
and it is also a $G$-variety when considering the action of $G$
by conjugation on $\hom(\Gamma,G)$.

The GIT quotient above is the the maximal spectrum of the ring $\mathbb{C}[\hom(\Gamma,G)]^{G}$
of $G$-invariant regular functions on $\hom(\Gamma,G)$: 
\[
\hom(\Gamma,G)\quot G\,:=Specmax\left(\mathbb{C}[\hom(\Gamma,G)]^{G}\right).
\]
The GIT quotient does not parametrize all orbits, since some of them
may not be distinguishable by invariant functions. In fact, it can
be shown (see, for example \cite{Mu}) that the conjugation orbits
of two representations $\rho,\rho':\Gamma\to G$ define the same point
in $\hom(\Gamma,G)\quot G$ if and only if their closures intersect:
$\overline{G\cdot\rho}\cap\overline{G\cdot\rho'}\neq\emptyset$ (in
either the Zariski or the complex topology coming from an embedding
$\hom(\Gamma,G)\hookrightarrow\mathbb{C}^{N}$). For detailed definitions
and properties of general character varieties, we refer to \cite{FL,Si1}.

In this article, we will be mostly concerned with the case when $\Gamma$
is a finitely generated free abelian group, $\Gamma=\mathbb{Z}^{r}$
for some natural number $r$, the rank of $\Gamma$. The corresponding
$G$-character varieties:
\[
\mathcal{M}_{\mathbb{Z}^{r}}G=\hom(\mathbb{Z}^{r},G)\quot G,
\]
have many interesting properties, as representations in $\hom(\mathbb{Z}^{r},G)$
can be naturally identified with $r$-tuples of group elements $(A_{1},\cdots,A_{r})\in G^{r}$
that pairwise commute: $A_{i}A_{j}=A_{j}A_{i}$, for all $i,j=1,\cdots,n$. 

When $K$ is a compact Lie group, the analogous space of representations
$\hom(\mathbb{Z}^{r},K)$ is of central importance in determining
the so-called moduli space of vacua of supersymmetric gauge theories
on a $r$-dimensional torus, as studied in \cite{KS,BFM} and others.

\subsection{Mixed Hodge structures}

On a compact Kähler manifold $X$ the complex cohomology satisfies
the Hodge decomposition $H^{k}(X,\mathbb{C})\cong\bigoplus_{p+q=k}H^{p,q}(X)$
which verifies $H^{q,p}(X)\cong\overline{H^{p,q}(X)}$. This decomposition
of $H^{k}(X,\mathbb{C})$, a \textit{pure Hodge structure }\emph{of
weight $k$,} can be described, equivalently, by a decreasing filtration:
\[
H^{k}(X,\mathbb{C})=F_{0}\supseteq F_{1}\supseteq\ldots\supseteq F_{k+1}=0
\]
satisfying $F_{p}\cap\overline{F_{q}}=0$ and $F_{p}\varoplus\overline{F_{q}}=H^{k}(X,\mathbb{C})$
for all $p+q=k+1$.

This notion can be generalized to quasi-projective algebraic varieties
$X$ over $\mathbb{C}$, possibly non-smooth and/or non-compact. Namely,
the complex cohomology of any such variety is also endowed with a
natural filtration, the \emph{Hodge filtration} $F$, and moreover,
there is a special second increasing filtration on the \emph{rational}
cohomology: 
\[
0=W^{-1}\subseteq\ldots\subseteq W^{2k}=H^{k}\left(X,\mathbb{Q}\right),
\]
the \emph{weight filtration} $W$, satisfying a compatibility condition
with respect to the Hodge filtration: the latter induces a filtration
on the weighted graded pieces of the former, that needs to be a pure
Hodge structure. The vector space $H^{k}(X,\mathbb{Q})$, together
with the filtrations $F$ and $W$, is the prototype of a \emph{mixed
Hodge structure} (MHS). We denote the \emph{graded pieces} of the
associated decomposition by 
\[
H^{k,p,q}(X):=Gr_{F}^{p}Gr_{p+q}^{W_{\mathbb{C}}}H^{k}(X,\mathbb{C}),
\]
where $W_{\mathbb{C}}$ stands for the complexified weight filtration.
Notice that, even though different filtrations may lead to isomorphic
graded pieces, for convenience, we sometimes refer to the collection
of these $H^{k,p,q}(X)$ as the MHS of $X$. For background and proofs
we refer to \cite{PS} and the original articles by Deligne \cite{De1,De2}. 

The above constructions can be reproduced for the \emph{compactly
supported }cohomology groups $H_{c}^{k}(X,\mathbb{C})$, yielding
an analogous decomposition: 
\[
H_{c}^{k,p,q}(X):=Gr_{F}^{p}Gr_{p+q}^{W_{\mathbb{C}}}H_{c}^{k}(X,\mathbb{C}).
\]
MHS in the compactly supported context have interesting connections
to number theory as illustrated, for example, in the Appendix of \cite{HRV}
by N. Katz. 

Mixed Hodge structures satisfy some nice properties, as follows. 
\begin{prop}
\label{prop:Mhstructures}Let $X$ and $Y$ be complex quasi-projective
varieties. Then, 

\begin{enumerate}
\item For all $k,p,q$, we have $H^{k,q,p}(X)\cong\overline{H^{k,p,q}(X)}$; 
\item The weight and Hodge filtrations are preserved by algebraic maps.
Therefore, so are mixed Hodge structures; 
\item The Hodge and weight filtrations are preserved by the Künneth isomorphism.
Therefore, so are mixed Hodge structures; 
\item The mixed Hodge structures are compatible with the cup product: 
\begin{eqnarray*}
H^{k,p,q}\left(X\right)\smile H^{k',p',q'}\left(X\right) & \hookrightarrow & H^{k+k',p+p',q+q'}\left(X\right);
\end{eqnarray*}
\item If $X$ is smooth of complex dimension $n$, mixed Hodge structures
are compatible with Poincaré duality: 
\begin{eqnarray*}
H^{k,p,q}\left(X\right) & \cong & \left(H_{c}^{2n-k,\,n-p,\,n-q}\left(X\right)\right)^{*}.
\end{eqnarray*}
\end{enumerate}
\end{prop}

\begin{proof}
All of these statements are standard. For convenience, we point to
appropriate references. The proof of (1) follows from the purity of
the Hodge structure on the graded pieces of the weight filtration.

The other proofs are found in chapters 4 to 6 of the book \cite{PS}.
Specifically, for (2) see \cite[Proposition 4.18]{PS}; (3) and (4)
appear in \cite[Theorem 5.44, Corollary 5.45]{PS} and (5) in \cite[Proposition 6.19]{PS}. 
\end{proof}

\subsection{Hodge polynomials and balanced varieties}

The spaces $H^{k,p,q}(X)$ are holomorphic invariants and encode important
geometric information (diffeomorphic complex manifolds may have non-isomorphic
mixed Hodge structures, as in Example \ref{exa:HT}).

The \emph{mixed Hodge numbers of $X$} are the complex dimensions
of the MHS pieces 
\[
h^{k,p,q}(X):=\dim_{\mathbb{C}}H^{k,p,q}(X),
\]
and are typically assembled in a polynomial. 
By definition, for a pure Hodge structure, $h^{k,p,q}(X)\neq0$ unless
$k=p+q$. 
\begin{defn}
Let $X$ be a complex quasi-projective variety of complex dimension
$d$. The \emph{mixed Hodge polynomial} of $X$ (also called \emph{Hodge-Deligne polynomial}) is the three variable polynomial of degree $\leq2d$ 
\begin{eqnarray*}
\mu_{X}\left(t,u,v\right) & := & \sum_{k.p,q\geq0}h^{k,p,q}(X)\,t^{k}u^{p}v^{q}.
\end{eqnarray*}
Its specialization for $t=-1$ 
\begin{eqnarray*}
E_{X}\left(u,v\right) & := & \sum_{k.p,q\geq0}h^{k,p,q}(X)\left(-1\right)^{k}u^{p}v^{q}
\end{eqnarray*}
is called the $E$-\emph{polynomial} of $X$. 
\end{defn}

\begin{rem}
(1) The specialization of $\mu_{X}$ for $u=v=1$ gives the \emph{Poincaré
polynomial} of $X$: 
\[
P_{X}(t):=\sum_{k\geq0}b_{k}(X)\,t^{k},
\]
with $b_{k}(X):=\dim_{\mathbb{C}}H^{k}(X,\mathbb{C})$ being the Betti
numbers of $X$. Note that the coefficients of $\mu_{X}$ and of $P_{X}$
are \emph{non-negative} integers, whereas $E_{X}$ lives in the ring
$\mathbb{Z}[u,v]$.\\
(2) As mentioned before, there is an entirely parallel theory for
the compactly supported cohomology. Here, the associated Hodge numbers
are denoted by $h_{c}^{k,p,q}:=\dim_{\mathbb{C}}H_{c}^{k,p,q}\left(X\right)$.
If $\mathcal{P}_{X}$ stands for one of the polynomials in the above
definition, we will distinguish its compactly supported version by
writing $\mathcal{P}_{X}^{c}$.\\
(3) \emph{Comment on terminology:} There are inconsistencies in the
literature on the terminology used for these polynomials. Since $h^{k,p,q}(X)$
are generally called Hodge-Deligne (or mixed Hodge) numbers, we refer
to $\mu_{X}$ as \emph{Hodge-Deligne} or \emph{mixed Hodge} \emph{polynomial}.
To emphasize the distinction, the \emph{compactly supported $E$-polynomial
$E_{X}^{c}$ }will also be called the \emph{Serre polynomial }of $X$\emph{,
}since its crucial behavior, as a generalized Euler characteristic,
was first used by Serre in connection with the Weil conjectures (see
\cite{Du}).\\
(4) Many specializations of the $E$-polynomial have been studied
in the literature. There is, for example, the \textsl{weight }\emph{polynomial}
$W_{X}\left(y\right):=\sum_{k,p}(-1)^{k}w^{k,p}(X)\,y^{p}$, using
the graded pieces of the weight filtration $w^{k,p}(X):=\dim_{\mathbb{C}}Gr_{p}^{W_{\mathbb{C}}}H^{k}(X,\mathbb{C})$
(see \cite{DL}). This is a specialization of the $E$-polynomial
since $W_{X}(y)=E_{X}(y,y)$. Also, Hirzebruch's \emph{$\chi_{y}$-genus}
and the \emph{signature} $\sigma$ of a complex manifold $X$ are
given, in terms of $E_{X}(u,v)$, as: $\chi_{y}(X)=E_{X}(-y,1)$ and
$\sigma(X)=E_{X}(-1,1)$, respectively (see Hirzebruch \cite{Hir}).
\end{rem}

We now collect some well known important properties of these polynomials,
for later use. 
\begin{prop}
\label{prop:polynomials}For a quasi-projective variety $X$, we have:

\begin{enumerate}
\item The polynomials $\mu_{X}$ and $E_{X}$ are symmetric in the variables
$u$ and $v$; in particular, if $h^{k,p,q}(X)\neq0$ then $h^{k,q,p}(X)\neq0$; 
\item Let $h^{k,p,q}(X)\neq0$. Then $p,q\leq k$. Moreover, if $X$ is
\emph{smooth} then $p+q\geq k$; if $X$ is \emph{projective} then
$p+q\leq k$. In particular, if $X$ is a compact Kähler manifold
$p+q=k$; 
\item The (topological) Euler characteristic $\chi(X)$ is given by $\chi(X)=E_{X}(1,1)$; 
\item The Serre polynomial (compactly supported $E$-polynomial) $E_{X}^{c}$,
is additive for stratifications of $X$ by locally closed subsets,
and its degree is equal to $2\dim_{\mathbb{C}}X$. 
\item All polynomials $\mu_{X}$, $P_{X}$ and $E_{X}$ are multiplicative
under cartesian products. 
\end{enumerate}
\end{prop}

\begin{proof}
(1) follows from item (1) of Proposition \ref{prop:Mhstructures}.
Item (2) is proved in \cite[Proposition 4.20]{PS} and \cite[Theorem 5.39]{PS}.
Item (3) is immediate from the definition. The proof of (4) can be
found in \cite[Corollary 5.57]{PS}, and (5) follows directly from
\ref{prop:Mhstructures}(4).
\end{proof}
A common feature of the varieties in this paper is that their mixed
Hodge structure is ``diagonal'': for each $k$, the only non-zero
mixed Hodge numbers are $h^{k,p,q}$ with $p=q$. 
\begin{defn}
A quasi-projective variety $X$ is said to be \emph{balanced} or of
\emph{Hodge-Tate} \emph{type} if for every non-negative integer $k\in\mathbb{N}_{0}$,
and all $p\neq q$, $h^{k,p,q}(X)=0$. In other words, if $h^{k,p,q}(X)\neq0$
then $q=p$. We call $p+q$ the \emph{total weight} of $H^{k,p,q}(X)$.
\end{defn}

\begin{example}
(1) If $X$ is connected, $H^{0}(X,\mathbb{C})\cong\mathbb{C}$ has
always a pure Hodge structure, with trivial decomposition $H^{0}(X,\mathbb{C})=H^{0,0,0}(X)$.
Dually, when $X$ is also smooth, the compactly supported cohomology
is also a trivial decomposition $H_{c}^{2n}\left(X,\mathbb{C}\right)=H_{c}^{2n,n,n}\left(X\right).$\\
\label{exa:C*}(2) $X=\mathbb{C}^{n}$ is a (non compact) Kähler manifold
with cohomology only in degree zero. By the above, it has trivial
pure Hodge structure: $H^{*}\left(X,\mathbb{C}\right)=H^{0}\left(X,\mathbb{C}\right)=H^{0,0,0}\left(X\right)$
and so 
\[
\mu_{\mathbb{C}^{n}}\left(t,u,v\right)=1,\quad\quad\mu_{\mathbb{C}^{n}}^{c}(t,u,v)=t^{2n}u^{n}v^{n},
\]
where the compactly supported version follows from Poincaré duality.\\
(3) Let $X=\mathbb{C}^{*}=\mathbb{C}\setminus\{0\}$. Although Kähler,
its cohomology has no pure Hodge structure, since $\dim_{\mathbb{C}}H^{1}(X,\mathbb{C})=1$.
Being smooth, using Proposition \ref{prop:Mhstructures}(1)-(2), the
only non-zero $h^{1,p,q}$ is $h^{1,1,1}$, so $h^{0,0,0}=h^{1,1,1}=1$
and all other Hodge numbers vanish. We then get 
\begin{eqnarray*}
\mu_{\mathbb{C}^{*}}\left(t,u,v\right) & = & 1+tuv,
\end{eqnarray*}
and the only non-zero $h_{c}^{k,p,q}$ are $h_{c}^{2,1,1}=h_{c}^{1,0,0}=1$,
by Poincaré duality. Thus, $\mu_{\mathbb{C}^{*}}^{c}(t,u,v)=t^{2}uv+t$
and $E_{\mathbb{C}^{*}}^{c}(u,v)=uv-1$. Observe that this is compatible
with the decomposition into locally closed subsets $\mathbb{C}^{1}=\mathbb{C}^{*}\sqcup\mathbb{C}^{0}$
as in Proposition \ref{prop:polynomials}(4). Hence, $\mathbb{C}^{n}$
and $(\mathbb{C}^{*})^{n}$ are examples of balanced varieties. For
a simple example of a non-balanced variety we can take an elliptic
curve, or any compact Riemann surface of positive genus.\\
(4) \label{exa:HT}Consider the total space $X$ of the trivial line
bundle over an elliptic curve $X\cong(\mathbb{C}/\Lambda)\times\mathbb{C}$,
where $\Lambda$ is a rank two lattice in $(\mathbb{C},+)$. It is
easy to see that $X$ is \emph{real analytically} isomorphic to $\left(\mathbb{C}^{*}\right)^{2}$
(but not complex analytically or algebraically isomorphic). From the
Künneth isomorphism and considerations analogous to Example \ref{exa:C*}
we get: 
\begin{eqnarray*}
\mu_{\left(\mathbb{C}^{*}\right)^{2}}\left(t,u,v\right) & = & \left(1+tuv\right)^{2}=1+2tuv+t^{2}u^{2}v^{2},\\
\mu_{X}\left(t,u,v\right) & = & (1+tu)(1+tv)=1+t(u+v)+t^{2}uv.
\end{eqnarray*}
Indeed, $\left(\mathbb{C}^{*}\right)^{2}$ is balanced whereas the
cohomology of $X$ is pure.
\end{example}

\begin{rem}
(1) The last example is a very special case (the genus 1, rank 1 case)
of the non-abelian Hodge correspondence mentioned in the Introduction,
which produces diffeomorphisms between (Zariski open subsets) of moduli
spaces of flat connections and certain moduli spaces of Higgs bundles
over a given Riemann surface. The fact that one diffeomorphism type
is balanced (the flat connection side of the correspondence) and the
other is pure is a general feature (see \cite{HRV,HLRV}).
\\
(2) \label{rem:balanced}If $X$ is balanced, its $E$-polynomial
depends only on the product $uv$, so it is common to adopt the change
of variables $x\equiv uv$. When written in this variable, the degree
of $E_{X}^{c}(x)$ is now equal to $\dim_{\mathbb{C}}X$, instead
of $2\dim_{\mathbb{C}}X$. 
\end{rem}

\section{Separably pure, elementary and round varieties}

In this section, we collect many properties of MHS that are necessary
later on. We also describe the types of Hodge structures that allow
the recovery of the mixed Hodge polynomial given the $E$- or the
Poincaré polynomial (Theorem \ref{thm:recover-mixed}), and concentrate
on the case of round varieties, which are the Hodge types of our character
varieties. We tried to be self-contained for the benefit of researchers
in the field of character varieties or Higgs bundles that may not
be familiar with MHS. 

\subsection{Elementary and separably pure varieties}

The mixed Hodge structures on the cohomology of a given quasi-projective
variety $X$ may be trivial, ie, the decomposition of every $H^{k}(X,\mathbb{C})$
is the trivial one, and many such examples are considered here. When
this happens, the only non-zero $h^{k,p,q}(X)$ satisfy $q=p$ (by
Proposition \ref{prop:polynomials}(1)) and much of what can be said
about the cohomology can be transported to mixed Hodge structures.
Adapting some notions from \cite{DL} (who worked with the weight
polynomial), we introduce the following terminology.
\begin{defn}
Let $X$ be a quasi-projective variety. $X$ (or its cohomology) is
called \emph{elementary} if its mixed Hodge structures are trivial
decompositions of the cohomology, so that for every $k\in\mathbb{\mathbb{N}}$
there is only one $p\in\mathbb{N}$ such that $h^{k,p,p}(X)\neq0$
(and $h^{k,p,q}(X)\neq0$ for $q\neq p$).\\
$X$ is said to be \emph{separably pure} if the mixed Hodge structure
on each $H^{k}\left(X,\mathbb{C}\right)$ is in fact pure of total
weight $w_{k}$, and such that $w_{j}\neq w_{k}$ for every $j\neq k$.
\end{defn}

\begin{rem}
(1) Note that $X$ is elementary if it is \emph{balanced} and there
is a \emph{weight function} $k\mapsto p_{k},$ (defined only for those
$k\in\mathbb{N}_{0}$ with $H^{k}(X,\mathbb{C})\neq0$) such that
$h^{k,p,q}=0$ for every pair $(p,q)$ not equal to $(p_{k},p_{k})$.
In this case:
\begin{equation}
Gr_{2p_{k}}^{W_{\mathbb{C}}}H^{k}(X,\mathbb{C})=H^{k,p_{k},p_{k}}\left(X\right)=H^{k}(X,\mathbb{C}).\label{eq:elementary-Hodge}
\end{equation}
A general weight function is not enough to recover $\mu_{X}$ from
the weight or the $E$-polynomials (different degrees of cohomology
may have equal total weights). However, this can be done (see Theorem
\ref{thm:recover-mixed}) if the weight function $k\mapsto p_{k}$
is injective, in which case the equality \eqref{eq:elementary-Hodge}
takes the stronger form: $\bigoplus_{m}Gr_{2p_{k}}^{W_{\mathbb{C}}}H^{m}(X,\mathbb{C})=H^{k,p_{k},p_{k}}\left(X\right)=H^{k}(X,\mathbb{C})$.\\
(2) In a pure Hodge structure of total weight $k$ on $H^{k}\left(X,\mathbb{Q}\right)$
the only non-zero weight summand is $Gr_{2k}^{W}H^{k}\left(X,\mathbb{Q}\right)$.
So, a pure total cohomology is separably pure, but not conversely,
as the case $\mathbb{C}^{*}$ shows (Example \ref{exa:C*}).\\
(3) When $X$ is separably pure, instead of the weight function, one
can define a \emph{degree function} $(p,q)\mapsto k=k(p,q)$ (defined
\emph{only} on pairs $(p,q)$ such that $h^{k,p,q}(X)\neq0$). Noting
that, in fact, the degree $k$ only depends on the total weight $p+q$
(being separably pure) we can write this as $(p,q)\mapsto k_{p+q}$. 
\end{rem}

In this article, most varieties are both separably pure and balanced,
and an alternative characterization follows. 
\begin{lem}
\label{lem:elm+inj}A quasi-projective variety $X$ is separably pure
and balanced if and only if it is elementary and its weight function
$k\mapsto p_{k}$ is injective.
\end{lem}

\begin{proof}
If $X$ is separably pure, the total weight in each $H^{k}(X,\mathbb{C})$
has to be constant. But if $X$ is also balanced, given $k$, all
$h^{k,p,q}(X)$ vanish except for a unique pair $(p,q)=(p_{k},p_{k})$,
so we have an assignment $k\mapsto p_{k}$ proving that $X$ is elementary.
Moreover, since the total weights are different for distinct $k$,
the weight function is injective. The converse statement is easy since
an elementary variety is Hodge-Tate and an injective weight function
implies injectivity for total weights. 
\end{proof}
\begin{example}
A family of balanced and pure varieties are the smooth projective
toric varieties (hence these are elementary). Indeed, every such toric
variety $X$, where the $\mathbb{C}^{*}$ action has $d_{j}$ orbits
in (complex) dimension $j=0,\cdots,\dim_{\mathbb{C}}X$, has (See
\cite[Ex. 5.58]{PS}): 
\[
\mu_{X}(t,u,v)=\sum_{j=0}^{n}d_{j}(t^{2}uv-1)^{j},
\]
and the weight function is $2j\mapsto j$ (the weights being $(j,j)$).
For example, the Hodge-Deligne polynomials of projective spaces are
$\mu_{\mathbb{P}_{\mathbb{C}}^{n}}\left(t,u,v\right)=\sum_{j=0}^{n}t^{2j}u^{j}v^{j}.$ 
\end{example}

As in the case of the complex affine multiplicative group $\mathbb{C}^{*}$,
more general complex affine algebraic groups are balanced, but not
necessarily pure or separably pure. 
\begin{example}
\label{exa:mHD-GLn}The Poincaré polynomial of $GL(n,\mathbb{C})$
is well known, given by\\
 $P_{GL(n,\mathbb{C})}(t)=\prod_{j=1}^{n}\left(1+t^{2j-1}\right)$.
Also, by \cite[Theorem 9.1.5]{De2}, we have: 
\begin{eqnarray*}
\mu_{GL(n,\mathbb{C})}\left(t,u,v\right) & = & \prod_{j=1}^{n}\left(1+t^{2j-1}u^{j}v^{j}\right).
\end{eqnarray*}
For example, $\mu_{GL(3,\mathbb{C})}\left(t,x\right)=1+tx+t^{3}x^{2}+t^{4}x^{3}+t^{5}x^{3}+t^{6}x^{4}+t^{9}x^{6}$,
(writing $x=uv$, see Remark \ref{rem:balanced}(2)). So, $GL(3,\mathbb{C})$
is elementary (hence balanced) but \emph{not separably pure}: both
degrees $4$ and $5$ have associated total weight $6$ (the terms
with $x^{3}$), so $GL(n,\mathbb{C})$ is not separably pure, for
$n\geq3$. Moreover, the same argument readily shows that $GL(n,\mathbb{C})$
is \emph{not elementary} for $n\geq5$. 
\end{example}

The above examples show that this ``yoga of weights'', as alluded
by Grothendieck, is very useful in understanding general properties
of certain classes of varieties. When we know that a particular variety
$X$ has a degree or a weight function as above, we can determine
the full collection of triples $(k,p,q)$, such that $h^{k,p,q}(X)\neq0$.
\begin{figure}
\includegraphics[scale=0.4]{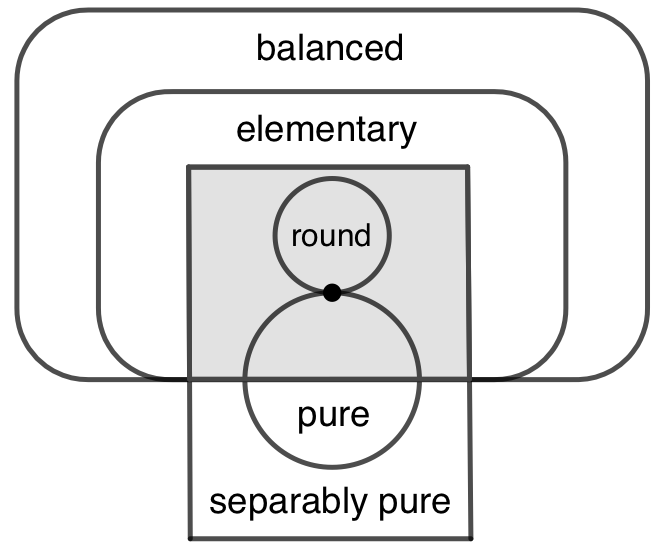}\label{fig:MHS}\caption{Venn diagram with several classes of MHS}
\end{figure}

In Figure \ref{fig:MHS}, the shaded area illustrates Lemma \ref{lem:elm+inj};
for the definition of round, see subsection 3.2, below. The next result
shows that elementary and separably pure are indeed the correct notions
to be able to determine the mixed Hodge polynomial from the Poincaré
or the $E$-polynomial, respectively.
\begin{thm}
\label{thm:recover-mixed}Let $X$ be a quasi-projective variety of
dimension $n$. Then: 

\begin{enumerate}
\item If $X$ is elementary, with known weight function, its Poincaré polynomial
determines its Hodge-Deligne polynomial. 
\item If $X$ is separably pure, with known degree function, its $E$-polynomial
determines its Hodge-Deligne polynomial. 
\end{enumerate}
\end{thm}

\begin{proof}
(1) Suppose the Poincaré polynomial of $X$ is $P_{X}\left(t\right)=\sum_{k}b_{k}t^{k}$
and the weight function is $k\mapsto(p_{k},p_{k})$. Then, since the
only non-trivial mixed Hodge pieces are $H^{k,p_{k},p_{k}}(X)$, we
get $\mu_{X}\left(t,u,v\right)=\sum_{k}b_{k}t^{k}u^{p_{k}}v^{p_{k}}$.

(2) Similarly, writing $E_{X}(u,v)=\sum_{p,q}a_{p,q}u^{p}v^{q}$,
and the degree function as $(p,q)\mapsto k_{p+q}$, since the total
weights are in one-to-one correspondence with the degrees of cohomology,
we obtain $\mu_{X}(t,u,v)=\sum_{k,p,q}a_{p,q}(-t)^{k_{p+q}}u^{p}v^{q}$. 
\end{proof}

\subsection{Round varieties}

From Theorem \ref{thm:recover-mixed}, if a variety $X$ is both balanced
\emph{and} separably pure, then $\mu_{X}$ can be recovered from either
$E_{X}$ or $P_{X}$, knowing their degree/weight functions. A specially
interesting case is the following.

\begin{defn}
\label{def:Round}Let $X$ be a quasi-projective variety. If the only
non-zero Hodge numbers are of type $h^{k,k,k}\left(X\right)$, $k\in\{0,\cdots,2\dim_{\mathbb{C}}X\}$,
we say that $X$ is \textsl{round}. 
\end{defn}

In other words, a round variety is both elementary and separably pure
and its only $k$-weights have the form $(k,k)$. Round varieties
are referred as ``minimaly pure'' balanced varieties in Dimca-Lehrer
(see \cite[Def. 3.1(iii)]{DL}). 

\begin{rem}
In general, cartesian products of elementary varieties are not elementary,
and similarly for separably pure varieties. For instance, using Example
\ref{exa:mHD-GLn} with $n=2$, we see that $GL(2,\mathbb{C})\times GL(2,\mathbb{C})$
is not separably pure. On the other hand, the following result holds
for round varieties.
\end{rem}

\begin{prop}
\label{prop:round-necessary}Let $X$ and $Y$ be round varieties.
Then: 

\begin{enumerate}
\item The Hodge-Deligne polynomial of $X$ reduces to a one variable polynomial,
and can be reconstructed from either the $E$ or the Poincaré polyonmial:
\begin{eqnarray*}
\mu_{X}\left(t,u,v\right) & = & P_{X}\left(tuv\right)=E_{X}\left(-tu,v\right).
\end{eqnarray*}
\item The cartesian product $X\times Y$ is round. 
\end{enumerate}
\end{prop}

\begin{proof}
(1) By definition, if $X$ is round, we can write: 
\[
\mu_{X}(t,u,v)=\sum_{k\geq0}h^{k,k,k}(X)\,t^{k}u^{k}v^{k},
\]
so that $\mu_{X}$ is a polynomial in $tuv$. Moreover, its Betti
numbers are $b_{k}(X)=h^{k,k,k}(X)$ giving the first equality. The
second follows from $E_{X}(x)=\sum_{k\geq0}(-1)^{k}h^{k,k,k}(X)\,x^{k}$.

(2) This follows at once from (1) and from Proposition \ref{prop:polynomials}(5).
\end{proof}
\begin{rem}
\label{rem:duality}(1) If $X$ satisfies Poincaré duality on MHS,
and $\dim_{\mathbb{C}}X=n$, one has 
\[
\mu_{X}^{c}(t,u,v)=(t^{2}uv)^{n}\mu_{X}({\textstyle \frac{1}{t}},{\textstyle \frac{1}{u}},{\textstyle \frac{1}{v}}),\quad\quad P_{X}^{c}(t)=t^{2n}P_{X}({\textstyle \frac{1}{t}}),\quad\quad E_{X}^{c}(u,v)=(uv)^{n}E_{X}({\textstyle \frac{1}{u},{\textstyle \frac{1}{v}}}).
\]
In particular $\chi(X)=E_{X}^{c}(1,1)$. If $X$ is additionally round,
analogously to Proposition \eqref{prop:round-necessary}, $\mu_{X}^{c}$
can be reconstructed from $P_{X}^{c}$ and $E_{X}^{c}$ as: 
\[
\mu_{X}^{c}(t,u,v)=(uv)^{-n}P_{X}^{c}(tuv)=(-t)^{n}\,E_{X}^{c}(-tu,v).
\]
(2) A sufficient condition for roundness is the following: If $X$
is balanced and separably pure and its cohomology \emph{has no gaps},
in the sense that for every $k\in\mathbb{N}$, the condition $H^{k}(X,\mathbb{C})\neq0$
implies $H^{k-1}(X,\mathbb{C})\neq0$, then $X$ is round. This is
easy to see from Lemma \ref{lem:elm+inj} and the restrictions on
weights (Proposition \ref{prop:polynomials}(2)). 
\end{rem}

\section{\textbf{Cohomology and Mixed Hodge Structures for Finite Quotients}}

Let $F$ be a finite group and $X$ a complex quasi-projective $F$-variety.
In this section, we outline some results on the cohomology and mixed
Hodge structures of quotients of the form $X^{r}/F$, where $F$ acts
diagonally on the cartesian product $X^{r}$, for general $r\geq1$.
Of special relevance is a formula, in Corollary \ref{cor:quotmhs},
for the Hodge-Deligne polynomial of $X^{r}/F$ for an elementary variety
$X$ whose cohomology is a simple exterior algebra.

\subsection{Equivariant mixed Hodge structures}

The MHS of the ordinary quotient $X/F$ is related with the one of
$X$ and its $F$-action, as follows. Since $F$ acts algebraically
on $X$, it induces an action on its cohomology ring preserving the
degrees, and, by Proposition \ref{prop:Mhstructures}(1), the mixed
Hodge structures. Therefore, $H^{k,p,q}\left(X\right)$ and $Gr_{p+q}^{W_{\mathbb{C}}}H^{k}\left(X,\mathbb{C}\right)$
are also $F$-modules. Denoting these by $\left[Gr_{p+q}^{W}H^{k}\left(X,\mathbb{C}\right)\right]$
and $\left[H^{k,p,q}\left(X\right)\right]$, and calling them \emph{equivariant
MHS}, one may codify this information in polynomials with coefficients
belonging to the representation ring of $F$, $R(F)$ (cf. \cite{DL}).
\begin{defn}
\label{def:equipol}The \textit{equivariant mixed Hodge polynomial}
is defined as:

\[
\mu_{X}^{F}\left(t,u,v\right)=\sum_{k,p,q}\left[H^{k,p,q}\left(X\right)\right]t^{k}u^{p}v^{q}\in R(F)\left[t,u,v\right].
\]
Evaluating at $t=-1$, gives us the \textit{equivariant E-polynomial}:
\[
E_{X}^{F}\left(u,v\right)=\sum_{k,p,q}(-1)^{k}\left[H^{k,p,q}\left(X\right)\right]u^{p}v^{q}\in R(F)\left[u,v\right].
\]
\end{defn}

As in the non-equivariant case, we adopt the change of variable $x=uv$
when $X$ is balanced. As in Proposition \ref{prop:Mhstructures},
several simple properties can be deduced. 
\begin{prop}
\label{prop:equMhstructures}Let $X$ be a quasi-projective $F$-variety,
for a finite group $F$, and let $\mathcal{P}_{X}^{F}$ be one of
the polynomials in Definition \ref{def:equipol}. Then 

\begin{enumerate}
\item $\mathcal{P}_{X}$ is obtained by replacing each representation in
$\mathcal{P}_{X}^{F}$ by its dimension; 
\item The Künneth formula and Poincaré Duality, for $X$ smooth, are compatible
with equivariant MHS:
\begin{eqnarray*}
\mathcal{P}_{X\times Y}^{F} & = & \mathcal{P}_{X}^{F}\,\varotimes\,\mathcal{P}_{Y}^{F}\\
\left[Gr_{p+q}^{W_{\mathbb{C}}}H^{k}\left(X,\mathbb{C}\right)\right] & = & \left[\left(Gr_{2n-\left(p+q\right)}^{W_{\mathbb{C}}}H_{c}^{2n-k}\left(X,\mathbb{C}\right)\right)^{*}\right]\\
\left[H^{k,p,q}\left(X\right)\right] & = & \left[\left(H_{c}^{2n-k,n-p,n-q}\left(X\right)\right)^{*}\right],
\end{eqnarray*}
where $\varotimes$ means that we take tensor products of graded $F$-representations.
\end{enumerate}
\end{prop}

\begin{proof}
(1) This follows immediately from the definition of dimension of representation.
For (2), it suffices to see the that the Künneth and Poincaré maps
are also morphisms in the category of $F$-modules, which is easily
checked.
\end{proof}

\subsection{Cohomology of finite quotients }

We recall some known facts concerning the usual and the compactly
supported cohomology of the quotient $X/F$. Consider its equivariant
cohomology, defined on rational cohomology by 
\[
H_{F}^{*}\left(X,\mathbb{Q}\right):=H^{*}\left(EF\times_{F}X,\mathbb{Q}\right),
\]
where $EF$ is the universal principal bundle over $BF$, the classifying
space of $F$, and $EF\times_{F}X$ is the quotient under the natural
action, which admits an algebraic map $EF\times_{F}X\stackrel{\pi}{\longrightarrow}X/F$.
Since $F$ is finite, so is the stabilizer of any point for the $F$
action, and the Vietoris-Begle theorem (see e.g. \textbf{\cite[page 344]{Sp}})
implies that the pullback $\pi^{*}:H^{*}\left(X/F,\mathbb{Q}\right)\to H^{*}\left(EF\times_{F}X,\mathbb{Q}\right)$
provides an isomorphism 
\begin{equation}
H^{*}\left(X/F,\mathbb{Q}\right)\simeq H_{F}^{*}\left(X,\mathbb{Q}\right).\label{eq:equico_equiv_quoco}
\end{equation}
Moreover, the fibration 
\[
\xymatrix{X\ar[r] & EF\times_{F}X\ar[r] & BF}
\]
has an induced Serre spectral sequence satisfying (see \cite{Mc},
for example) 
\[
E_{2}^{p,q}\cong H^{p}\left(BF,\,H^{q}\left(X\right)\right)\cong H^{p}\left(F,\,H^{q}\left(X\right)\right)\Rightarrow H_{F}^{p+q}\left(X,\mathbb{Q}\right).
\]
Since $F$ is finite, one can deduce $H^{p}\left(F,\,H^{q}\left(X\right)\right)=0$
for all $p>0$ and all $q$, since sheaf cohomology vanishes in degrees
higher than $\dim F$. Then the Serre spectral sequence converges
at the second step, and this gives (\cite{Br})\textit{}\footnote{For polynomials, the superscript $F$ means we are taking the equivariant
version, for vector spaces with an $F$ action, the superscript denotes
the fixed subspace.} 
\begin{equation}
H_{F}^{*}\left(X,\mathbb{Q}\right)\simeq H^{*}\left(X,\mathbb{Q}\right)^{F}.\label{eq:equico_equiv_invco}
\end{equation}
Combining Equations \eqref{eq:equico_equiv_quoco} and \eqref{eq:equico_equiv_invco},
one gets an isomorphism of graded vector spaces:\footnote{We thank Donu Arapura for a suggestion leading to a shorter proof
of Proposition \ref{prop:PD-compact}}
\begin{eqnarray}
H^{*}\left(X/F,\mathbb{Q}\right) & \cong & H^{*}\left(X,\mathbb{Q}\right)^{F}.\label{eq:equichm}
\end{eqnarray}

\begin{prop}
\label{prop:PD-compact}Let $F$ be a finite group and $X$ a smooth
quasi-projective $F$-variety. Then, the pullback of the quotient
map $\pi^{*}:H^{*}\left(X/F,\mathbb{Q}\right)\to H^{*}\left(X,\mathbb{Q}\right)$
is injective and has $H^{*}\left(X,\mathbb{Q}\right)^{F}$ as its
image.
\end{prop}

\begin{proof}
Assume first that $F$ acts freely on $X$. Then, $X/F$ has a well
defined manifold structure, and one can realize the pullback in cohomology
by the pullback in differential forms. In particular, this shows that
the image of the pullback $\pi^{*}:H^{*}\left(X/F,\mathbb{Q}\right)\to H^{*}\left(X,\mathbb{Q}\right)$
is given by $H^{*}\left(X,\mathbb{Q}\right)^{F}$. Using \eqref{eq:equichm},
this means that the pullback map is bijective onto $H^{*}\left(X,\mathbb{Q}\right)^{F}$.

If $F$ does not act freely, the same argument can be reproduced for
the de Rham orbifold cohomology, in which representatives of orbifold
cohomology classes are sections of exterior powers of the orbifold
cotangent bundle (see \cite{ALR}). The result then follows because,
for manifolds such as $X$, the de Rham orbifold cohomology reduces
to the usual de Rham cohomology. 
\end{proof}
The isomorphism of \eqref{eq:equichm} can be obtained as the pullback
of the algebraic map $\pi:X\to X/F$. Given that pullbacks of algebraic
maps preserve mixed Hodge structures, we see that this isomorphism
respects MHS (see also \cite{LMN})
\begin{equation}
H^{*,*,*}\left(X,\mathbb{Q}\right)^{F}\cong H^{*,*,*}\left(X/F,\mathbb{Q}\right).\label{eq:isom-mhs}
\end{equation}
Moreover, since orbifolds satisfy Poincaré duality (see Satake \cite{Sa},
where these are called $V$-manifolds), this isomorphism is also valid
for the compactly supported cohomology.
\begin{cor}
\label{cor:quofrmeq}Let $X$ be a smooth complex quasi-projective
$F$-variety, for $F$ a finite group, and $\mathcal{P}_{X}\left(t,u,v\right)$
denote either the Poincaré, Hodge-Deligne or $E$-polynomials, for
the usual or the compactly supported cohomologies, Then, $\mathcal{P}_{X/F}\left(t,u,v\right)$
equals the coefficient of the trivial representation in $\mathcal{P}_{X}^{F}\left(t,u,v\right)$. 
\end{cor}

\begin{proof}
Given Equation \eqref{eq:isom-mhs}, the cohomology of $X/F$ coincides
with the invariant part of the induced action $F\curvearrowright H^{*}(X,\mathbb{C})$,
and this is precisely the coefficient of the trivial representation
in each equivariant polynomial $\mathcal{P}_{X}^{F}$.
\end{proof}
\begin{cor}
\label{cor:round-quot}Let $X$ be a smooth complex quasi-projective
$F$-variety, as above. If $h^{k,p,q}(X/F)\neq0$ then $h^{k,p,q}(X)\neq0$.
Consequently, if $X$ is balanced, separably pure or round, then the
same is true for $X/F$.
\end{cor}

\begin{proof}
By Corollary \ref{cor:quofrmeq}, $H^{k,p,q}(X/F)\cong H^{k,p,q}(X)^{F}\subset H^{k,p,q}(X)$,
so the first sentence follows. Since all the properties of being balanced,
etc, are relations between the coefficients of $t$ and $u,v$, they
survive to the quotient. 
\end{proof}

\subsection{Character formulae}

For a $F$-variety $X$, is useful to consider the characters of the
representations $H^{k,p,q}(X)$, when this space is viewed as a $F$-module.

For this, let $A_{g}\in\mbox{Aut}(H^{k,p,q}(X))$ be the induced automorphism
of $H^{k,p,q}(X)$ given by the action of an element $g\in F$. Given
$k$-weights $(p,q)$, denote by 
\begin{eqnarray*}
\chi_{k,p,q}:\,F & \to & \mathbb{C}\\
g & \mapsto & \mbox{tr}(A_{g})
\end{eqnarray*}
the character of $H^{k,p,q}(X)$. In general, if we denote the character
of a $F$-module $V$ by $\chi_{V}$, because of the properties of
these with respect to direct sums, we have: 
\begin{equation}
\chi_{\mu_{X}^{F}(t,u,v)}(g)=\sum_{k,p,q}\chi_{k,p,q}(g)\,t^{k}u^{p}v^{q},\label{eq:character-equivariant}
\end{equation}
where $\mu_{X}^{F}(t,u,v)$ is viewed as a $F$-module, and equivalently
as a direct sum of modules graded according to the triples $(k,p,q)$.
Let $|F|$ be the cardinality of $F$. 
\begin{thm}
\label{thm:coeff1} Let $X$ be a quasi-projective $F$-variety. Then
\begin{eqnarray*}
\mu_{X/F}\left(t,u,v\right) & = & \frac{1}{\left|F\right|}\sum_{g\in F}\sum_{k,p,q}\chi_{k,p,q}(g)\,t^{k}u^{p}v^{q}.
\end{eqnarray*}
\end{thm}

\begin{proof}
If $V$ is a $F$-module, and $V=\oplus_{i}V_{i}$ is a decomposition
of $V$ into irreducible sub-representations, then by the Schur orthogonality
relations, the coefficient of the trivial one-dimensional representation
$\mathbf{1}$ is given by: 
\[
\left\langle \chi_{V},\chi_{\mathbf{1}}\right\rangle =\frac{1}{\left|F\right|}\sum_{g\in F}\chi_{V}\left(g\right)\overline{\chi_{\mathbf{1}}\left(g\right)}=\frac{1}{\left|F\right|}\sum_{g\in F}\chi_{V}\left(g\right).
\]
Applying this to $V=\mu_{X}^{F}(t,u,v)$ gives, in view of Corollary
\ref{cor:quofrmeq}: 
\[
\mu_{X/F}\left(t,u,v\right)=\left\langle \chi_{\mu_{X}^{F}(t,u,v)},\,\chi_{\mathbf{1}}\right\rangle =\frac{1}{\left|F\right|}\sum_{g\in F}\chi_{\mu_{X}^{F}(t,u,v)}\left(g\right),
\]
and the wanted formula follows from Equation \eqref{eq:character-equivariant}.
\end{proof}
\begin{example}
Let $X=\mathbb{P}_{\mathbb{C}}^{1}\times\mathbb{P}_{\mathbb{C}}^{1}$
and consider the natural permutation action of $S_{2}\cong\mathbb{Z}_{2}$.
If $\mathbf{S}$ denotes the one dimensional sign representation,
by describing the induced action on cohomology, it is not difficult
to show that $\left[H^{0,0,0}(X)\right]\cong\left[H^{4,2,2}(X)\right]\cong\mathbf{1}$
(the trivial one-dimensional representation) and that $\left[H^{2,1,1}(X)\right]\cong\mathbf{1}\oplus\mathbf{S}$,
giving: 
\begin{eqnarray}
\mu_{X}^{S_{2}}\left(t,u,v\right) & = & \left[H^{0,0,0}(X)\right]\oplus\left[H^{2,1,1}(X)\right]t^{2}uv\,\oplus\left[H^{4,2,2}(X)\right]t^{4}u^{2}v^{2}\nonumber \\
 & = & \mathbf{1}\,\oplus\,\left(\mathbf{1}\oplus\mathbf{S}\right)t^{2}uv\,\oplus\,\mathbf{1}\,t^{4}u^{2}v^{2}.\label{eq:P1xP1}
\end{eqnarray}
Alternatively, writing $S_{2}=\{\pm1\}$, and taking the trivial characters
$\chi_{0,0,0}(g)=\chi_{4,2,2}(g)\equiv1$, for $g\in S_{2}$, and
$\chi_{2,1,1}(1)=2$, $\chi_{2,1,1}(-1)=0$, we can use Theorem \ref{thm:coeff1}
to get: 
\[
\mu_{X/S_{2}}\left(t,u,v\right)=\frac{1}{2}(1+2t^{2}uv+t^{4}u^{2}v^{2})+\frac{1}{2}(1+t^{4}u^{2}v^{2})=1+t^{2}uv+t^{4}u^{2}v^{2},
\]
which coincides with the coefficient of $\mathbf{1}$ in Equation
\eqref{eq:P1xP1}. Naturally, this is the expected polynomial, because
$X/S_{2}=\sym^{2}\left(\mathbb{P}_{\mathbb{C}}^{1}\right)\simeq\mathbb{P}_{\mathbb{C}}^{2}$. 
\end{example}

An interesting application of Theorem \ref{thm:coeff1}, is when the
cohomology of $X$ is an exterior algebra. To be precise, we say that
$H^{*}(X,\mathbb{C})$ is an \emph{exterior algebra of odd degree
$k_{0}$} if: 
\[
H^{k_{0}l}(X,\mathbb{C})\cong{\textstyle \bigwedge^{\!l}}H^{k_{0}}(X,\mathbb{C}),\quad\forall l\geq0,
\]
and all other cohomology groups are zero. 
\begin{cor}
\label{cor:quotmhs}Let $X$ be an elementary $F$-variety whose cohomology
is an exterior algebra of odd degree $k_{0}$, and let $H^{k_{0}}(X,\mathbb{C})=H^{k_{0},p_{0},p_{0}}(X)$
be its (trivial) mixed Hodge decomposition, for some $p_{0}\leq k_{0}$.
Then, for $r>0$ and the diagonal action of $F$ on $X^{r}$: 
\begin{eqnarray*}
\mu_{X^{r}/F}\left(t,x\right) & = & \frac{1}{\left|F\right|}\sum_{g\in F}\det\left(I+t^{k_{0}}x^{p_{0}}\,A_{g}\right)^{r},
\end{eqnarray*}
with $x=uv$, where $A_{g}$ is the automorphism of $H^{k}\left(X,\mathbb{C}\right)$
corresponding to $g\in F$, and $I$ is the identity automorphism.
In particular, if $X$ is round: 
\begin{eqnarray*}
\mu_{X^{r}/F}\left(t,x\right) & = & \frac{1}{\left|F\right|}\sum_{g\in F}\det\left(I+(tx)^{k_{0}}\,A_{g}\right)^{r}.
\end{eqnarray*}
\end{cor}

\begin{proof}
First, let $r=1$. Since $X$ is elementary and tensor and exterior
products preserve mixed Hodge structures, we get for all $l\geq0$,
\[
H^{lk_{0}}(X,\mathbb{C})={\textstyle \bigwedge^{\!l}}H^{k_{0}}(X,\mathbb{C})={\textstyle \bigwedge^{\!l}}H^{k_{0},p_{0},p_{0}}(X)=H^{lk_{0},lp_{0},lp_{0}}(X).
\]
Applying Theorem \ref{thm:coeff1} to this case, using $x=uv$, we
get 
\begin{eqnarray}
\mu_{X/F}\left(t,x\right) & = & \frac{1}{\left|F\right|}{\textstyle \sum_{g\in F}\sum_{l\geq0}}\ \chi_{lk_{0},lp_{0},lp_{0}}(g)\,t^{lk_{0}}x^{lp_{0}}.\label{eq:mu-de-X/F}
\end{eqnarray}
Now, for a general $F$-module $V$, with $g\in F$ acting as $V_{g}\in Aut(V)$,
we have: 
\[
\sum_{l\geq0}\chi_{\bigwedge^{l}V}(g)\,s^{l}=\det(I+s\,V_{g}).
\]
This can be seen by expanding the characteristic polynomial of $V_{g}$
in terms of traces of $\bigwedge^{l}V_{g}$ (see for example \cite[pg. 69]{Se}).
Substituting the last equality into \eqref{eq:mu-de-X/F}, we get
the result with $s=t^{k_{0}}x^{p_{0}}$. Now, for a general $r\geq1$,
it follows from Proposition \ref{prop:equMhstructures}(2) that for
the diagonal action $\mu_{X^{r}}^{F}=\left(\mu_{X}^{F}\right)^{r}$,
so 
\begin{eqnarray*}
\mu_{X^{r}/F}\left(t,x\right) & = & \frac{1}{\left|F\right|}{\textstyle \sum_{g\in F}}\ \chi_{\left(\mu_{X}^{F}\left(t,x\right)\right)^{r}}\left(g\right)\\
 & = & \frac{1}{\left|F\right|}{\textstyle \sum_{g\in F}}\left(\chi_{\mu_{X}^{F}\left(t,x\right)}\left(g\right)\right)^{r}\\
 & = & \frac{1}{\left|F\right|}{\textstyle \sum_{g\in F}}\det\left(I+t^{k_{0}}x^{p_{0}}A_{g}\right)^{r}.
\end{eqnarray*}
Finally, the round case follows by setting $p_{0}=k_{0}$. 
\end{proof}

\section{Abelian character varieties and their Hodge-Deligne polynomials}

In this section, we apply the previous formulae to the computation
of the Hodge-Deligne, Poincaré and $E$-polynomials, of the distinguished
irreducible component of some families of character varieties. The
important case of $GL(n,\mathbb{C})$-character varieties leads to
the action of the symmetric group on a torus, and is naturally related
to works of I. G. Macdonald \cite{Ma} and of J. Cheah \cite{Ch}
on symmetric products.

\subsection{Mixed Hodge polynomials of abelian character varieties}

As in Section 2.1, let $G$ be a connected complex affine reductive
group. For simplicity, the $G$-character variety of $\Gamma=\mathbb{Z}^{r}$,
a rank $r$ free abelian group, will be denoted by
\[
\mathcal{M}_{r}G:=\mathcal{M}_{\mathbb{Z}^{r}}G=\hom(\mathbb{Z}^{r},G)\quot G.
\]
In general, the varieties $\mathcal{M}_{r}G$ (as well as $\hom(\mathbb{Z}^{r},G)$)
are not irreducible. But there is a unique irreducible subvariety
containing the identity representation, that we call the distinguished
component and denote by $\mathcal{M}_{r}^{0}G$, which is constructed
as the image under the composition
\begin{equation}
T^{r}=\hom(\mathbb{Z}^{r},T)\stackrel{\iota}{\hookrightarrow}\hom(\mathbb{Z}^{r},G)\stackrel{\pi}{\to}\mathcal{M}_{r}G,\label{eq:pi}
\end{equation}
where $\pi$ is the GIT projection, and $T$ is a fixed maximal torus
of $G$. This image, 
\[
\mathcal{M}_{r}^{0}G:=(\pi\circ\iota)(\hom(\mathbb{Z}^{r},T)),
\]
is then a closed subvariety of $\mathcal{M}_{r}G$ (see \cite{Th})
that we call the \emph{distinguished component}. Let $W$ be the Weyl
group of $G$, acting by conjugation on $T$. We quote the following
result from \cite{Si2}. As in the introduction, denote by $\mathcal{M}_{r}^{\star}G$
the normalization\textbf{ }of $\mathcal{M}_{r}^{0}G$\textbf{ }as
an algebraic variety, so that there is a birational map $\nu:\mathcal{M}_{r}^{\star}G\to\mathcal{M}_{r}^{0}G$.
\begin{prop}
\label{Prop:irred-comp} \cite[Theorem 2.1]{Si2} Let $G$ be a complex
reductive group and $r\geq1$. Then, $\mathcal{M}_{r}^{0}G$ is an
irreducible component of $\mathcal{M}_{r}G$, and there is an isomorphism
$\mathcal{M}_{r}^{\star}G\cong T^{r}/W$. 
\end{prop}

We now prove Theorem \ref{thm:mHDP}.
\begin{thm}
\label{thm:mHs-General}Let $G$ be a complex reductive algebraic
group. Then, $\mathcal{M}_{r}^{\star}G$ is round and 
\begin{equation}
\mu_{\mathcal{M}_{r}^{\star}G}\left(t,u,v\right)=\frac{1}{|W|}\sum_{g\in W}\left[\det\left(I+tuv\,A_{g}\right)\right]^{r},\label{eq:mu-general}
\end{equation}
where $A_{g}$ is the automorphism of $H^{1}(T,\mathbb{C})$ given
by $g\in W$, and $I$ is the identity.
\end{thm}

\begin{proof}
Since cartesian products of round varieties are round, and the maximal
torus of $G$ is isomorphic to $(\mathbb{C}^{*})^{n}$ for some $n$,
$T$ is a round variety and has an algebraic action of $W$. Then
$W$ also acts diagonally on $T^{r}=\left(\mathbb{C}^{*}\right)^{nr}$,
so $T^{r}/W$ is also round by Corollary \ref{cor:round-quot}. 
Moreover, the cohomology of $T$ is an exterior algebra of degree
$k_{0}=1$, so Corollary \ref{cor:quotmhs} immediately gives the
desired formula for $T^{r}/W$. The Theorem follows from the isomorphism
$\mathcal{M}_{r}^{\star}G\cong T^{r}/W$ of Proposition \ref{Prop:irred-comp}.
\end{proof}
\begin{rem}
By Remark \ref{rem:duality}(1), we obtain, in the compactly supported
case:
\begin{equation}
\mu_{\mathcal{M}_{r}^{\star}G}^{c}\left(t,u,v\right)=\frac{t^{r\cdot\dim T}}{|W|}\sum_{g\in W}\left[\det\left(tuvI+\,A_{g}\right)\right]^{r},\label{eq:dual}
\end{equation}
where $\dim T$ is the rank of $G$. We also obtain a formula for
the Poincaré and for the Serre polynomial $E_{\mathcal{M}_{r}^{0}G}^{c}$
of the distinguished component $\mathcal{M}_{r}^{0}G$.
\end{rem}

\begin{thm}
\label{thm:E-poly}For every complex reductive algebraic group $G$
and $r\geq1$, the Poincaré polynomial (respectively the Serre polynomial)
of $\mathcal{M}_{r}^{0}G$ is given by substituting $u=v=1$ in Equation
\eqref{eq:mu-general} (respectively $t=-1$ in \eqref{eq:dual}).
\end{thm}

\begin{proof}
From \cite[Cor. 4.9]{FL}, there is a strong deformation retraction
from $\mathcal{M}_{r}^{0}G$ to $\hom^{0}(\mathbb{Z}^{r},K)/K$, the
path component of the space of commuting $r$-tuples of elements in
$K$, containing the trivial $r$-tuple, up to conjugation, where
$K$ is a maximal compact subgroup of $G$. Hence, these spaces have
the same Poincaré polynomials. On the other hand, the formula of \cite[Thm. 1.4]{St}
is the same as Equation \eqref{eq:mu-general} with $u=v=1$. Indeed,
given the identification of $H^{1}(T,\mathbb{C})$ with $\mathfrak{t}\cong\mathbb{C}^{n}$,
the Lie algebra of $T$: 
\[
H^{1}(T,\mathbb{C})\cong H^{1}\left(\left(\mathbb{C}^{*}\right)^{n},\mathbb{C}\right)=H^{1}\left(\mathbb{C}^{*},\mathbb{C}\right)^{n}\cong\mathfrak{t},
\]
and the fact that every cohomology class has a left invariant representative,
the action of the Weyl group $W=S_{n}$ on $H^{1}(T,\mathbb{C})$
coincides with the one used in \cite{St}, in the context of compact
Lie groups.\footnote{We thank S. Lawton for calling our attention to the recent preprint
\cite{St}, where the Poincaré polynomial of analogous spaces for
\emph{compact} Lie groups is computed by quite different methods.}

As indicated in Proposition \ref{prop:polynomials}(4), the Serre
polynomial ($E^{c}$-polynomial) is additive for disjoint unions of
locally closed subvarieties. Therefore, for every bijective normalization
morphism between algebraic varieties $f:X\to Y$ the $E^{c}$-polynomials
of $X$ and of $Y$ coincide. In particular, the $E^{c}$-polynomials
of $\mathcal{M}_{r}^{\star}G$ and $\mathcal{M}_{r}^{0}G$ coincide. 
\end{proof}

\subsection{Normality of the distinguished component }

Given the equalities of both Poincaré and Serre polynomials of $\mathcal{M}_{r}^{0}G$
and $\mathcal{M}_{r}^{\star}G$, it is interesting to check where
there is also an equality $\mu_{\mathcal{M}_{r}^{\star}G}=\mu_{\mathcal{M}_{r}^{0}G}$.
To handle this question, we start by considering some sufficient conditions
for normality of $\mathcal{M}_{r}^{0}G$. 
\begin{lem}
\label{lem:normal-covers}Let $F\subset G$ be a finite subgroup of
the center of $G$, and $H=G/F$. If $\mathcal{M}_{r}^{0}G$ is normal
then $\mathcal{M}_{r}^{0}H$ is normal.
\end{lem}

\begin{proof}
Let $\hom^{0}(\mathbb{Z}^{r},G):=\pi^{-1}(\mathcal{M}_{r}^{0}G)$.
By definition of $\pi$ in equation \eqref{eq:pi} this is the variety
of homomorphisms that can be conjugated, in $G$, to some representation
inside the maximal torus $T\subset G$ (and similarly for $\pi_{H}:\hom^{0}(\mathbb{Z}^{r},H)\to\mathcal{M}_{r}^{0}H$).
The fibration of algebraic groups $F\to G\to H,$ induces the following
commutative diagram:
\[
\begin{array}{ccccc}
\hom(\mathbb{Z}^{r},F) & \hookrightarrow & \hom(\mathbb{Z}^{r},T) & \twoheadrightarrow & \hom(\mathbb{Z}^{r},T')\\
\bigparallel &  & \downarrow &  & \downarrow\\
\hom(\mathbb{Z}^{r},F) & \hookrightarrow & \hom^{0}(\mathbb{Z}^{r},G) & \stackrel{\phi}{\twoheadrightarrow} & \hom^{0}(\mathbb{Z}^{r},H)\\
 &  & \quad\downarrow\pi &  & \quad\downarrow\pi_{H}\\
 &  & \mathcal{M}_{r}^{0}G & \twoheadrightarrow & \mathcal{M}_{r}^{0}H,
\end{array}
\]
where $T':=T/F$ is a maximal torus of $H$, and the surjections on
the two top rows are discrete fibrations (and finite étale morphisms).
Since $F$ is central, conjugating by $G$ or by $H$ are equivalent,
so that $\mathcal{M}_{r}^{0}G\cong\hom^{0}(\mathbb{Z}^{r},G)/H$,
and the map $\phi$ is $H$-equivariant, we obtain isomorphisms:
\[
\mathcal{M}_{r}^{0}H\cong\hom^{0}(\mathbb{Z}^{r},H)/H\cong\left(\hom^{0}(\mathbb{Z}^{r},G)/F^{r}\right)/H\cong\mathcal{M}_{r}^{0}G/F^{r}
\]
because the actions of $F^{r}\cong\hom(\mathbb{Z}^{r},F)$ and conjugation
by $H$ on commute. Hence, as an algebraic quotient by a finite group,
if $\mathcal{M}_{r}^{0}G$ is normal, then so is $\mathcal{M}_{r}^{0}F$.%
\end{proof}
\begin{rem}
If the action of $F^{r}\cong\hom(\mathbb{Z}^{r},F)$ on $\mathcal{M}_{r}^{0}G$
was free, the quotient $\mathcal{M}_{r}^{0}G\to\mathcal{M}_{r}^{0}H$
would also be étale, and the converse of Lemma \eqref{lem:normal-covers}
would be valid (see e.g. \cite[Theorem 4.4 (i)]{Dr}). However, this
is not the case for the isogeny $Spin(n,\mathbb{C})\to SO(n,\mathbb{C})$
($Spin(n,\mathbb{C})$ is the complexification of the universal cover
$Spin(n)\to SO(n)$, of the compact group $SO(n)$). The normality
of $\mathcal{M}_{r}^{0}Spin(n,\mathbb{C})$ and of $\mathcal{M}_{r}^{0}G$
for exceptional groups $G$ is known to be a difficult problem (see
Sikora \cite[Problem 2.3]{Si2}).
\end{rem}

For a complex reductive group $G$, let $DG=[G,G]$ denote its derived
group, which is a semisimple group. Let us call \emph{classical semisimple
group} to a group $G$ that is a direct product of groups of the three
classical families: $SL(n,\mathbb{C})$, $Sp(n,\mathbb{C})$ and $SO(n,\mathbb{C})$,
$n\in\mathbb{N}$. 
\begin{lem}
\label{lem:dg-classical}If $DG$ is a classical semisimple group,
then $\mathcal{M}_{r}^{0}G$ is a normal variety.
\end{lem}

\begin{proof}
Sikora proved that, when $G=SL(n,\mathbb{C})$, $Sp(n,\mathbb{C})$
or $SO(n,\mathbb{C})$, there are algebraic isomorphisms $\mathcal{M}_{r}^{0}G\cong T^{r}/W$
\cite{Si2}. So for $G$ in these 3 families, $\mathcal{M}_{r}^{0}G$
is normal. It is clear that if $G=G_{1}\times G_{2}$ then the maximal
torus is also a product, and that $\mathcal{M}_{r}^{0}G=\mathcal{M}_{r}^{0}G_{1}\times\mathcal{M}_{r}^{0}G_{2}$.
Thus, $\mathcal{M}_{r}^{0}G$ is normal for any classical semisimple
group $G$.

Finally, the result follows from Lemma \ref{lem:normal-covers}, by
taking finite quotients. Indeed, by the central isogeny theorem, any
reductive group $G$ is a finite central quotient of a product of
its derived group $DG$ with a torus $T$ (and clearly $\mathcal{M}_{r}T=\mathcal{M}_{r}^{0}T\cong T^{r}$).
\end{proof}
Since the hypothesis of Lemma \ref{lem:dg-classical} implies that
$\mathcal{M}^{\star}G=\mathcal{M}_{r}^{0}G$, we have proved the following.
\begin{thm}
\label{thm:main2}Let $r\geq1$, and let $G$ be a reductive group
whose derived group is a classical group. Then, the mixed Hodge polynomial
of $\mathcal{M}_{r}^{0}G$ is given by formula \eqref{eq:mu-general}.
\end{thm}

This motivates the following conjecture.
\begin{conjecture}
For every $r\geq1$ and complex reductive $G$, formula \eqref{eq:mu-general}
holds for $\mathcal{M}_{r}^{0}G$.
\end{conjecture}

\subsection{The $GL(n,\mathbb{C})$ and $SL(n,\mathbb{C})$ cases}

The case of $G=GL(n,\mathbb{C})$ is instructive, where the Weyl group
is just the symmetric group, denoted by $S_{n}$. If $X$ is a variety,
we denote its $n$-fold symmetric product by $X^{(n)}$ or by $\sym^{n}(X)=X^{n}/S_{n}$.
As a set, $\sym^{n}(X)$ is the set of unordered $n$-tuples of (not
necessarily distinct) elements of $X$. 
\begin{prop}
\label{prop:GLn-char-var}Let $G=GL(n,\mathbb{C})$, and let $T\cong(\mathbb{C}^{*})^{n}$
denote a maximal torus of $G$. Then $\mathcal{M}_{r}G=\mathcal{M}_{r}^{0}G=\mathcal{M}_{r}^{\star}G$
and we have isomorphisms of affine algebraic varieties 
\[
\mathcal{M}_{r}G\,\cong T^{r}/S_{n}\cong\sym^{n}(\mathbb{C}^{*})^{r}.
\]
\end{prop}

\begin{proof}
In \cite[Cor 5.14]{FL}, it was shown that $\mathcal{M}_{r}G=\hom(\mathbb{Z}^{r},GL(n,\mathbb{C}))\quot GL(n,\mathbb{C})$
is an irreducible variety (It is also path connected, given the strong
deformation retraction from $\mathcal{M}_{r}G$ to the path connected
compact space $\hom(\mathbb{Z}^{r},U(n))/U(n)$, see \cite{KS,FL}).
Since $\mathcal{M}_{r}^{0}G$ is irreducible of the same dimension
(by \cite[Thm 2.1(1)]{Si2}), $\mathcal{M}_{r}G=\mathcal{M}_{r}^{0}G$.
Moreover, $\mathcal{M}_{r}^{0}G$ is normal, hence isomorphic to $\mathcal{M}_{r}^{\star}G$,
and so $\mathcal{M}_{r}G\cong\mathcal{M}_{r}^{\star}G\cong T^{r}/W$,
by Sikora's results in \cite[Thm 2.1(2)-(3)]{Si2}). Since $W\cong S_{n}$
acts diagonally, we have finally $T^{r}/W=((\mathbb{C}^{*})^{n})^{r}/S_{n}\cong((\mathbb{C}^{*})^{r})^{n}/S_{n}=\sym^{n}(\mathbb{C}^{*})^{r}$.
\end{proof}
We now turn to the proof of Theorem \ref{thm:E-SLn}, on the $SL(n,\mathbb{C})$-character
variety of $\Gamma=\mathbb{Z}^{r}$, and start with the case $G=GL(n,\mathbb{C})$.
Let $M_{\sigma}$ denote a $n\times n$ permutation matrix (in some
basis) corresponding to $\sigma\in S_{n}$ and let $I_{n}$ be the
$n\times n$ identity matrix. 

\begin{prop}
\label{prop:mHs-GLn-det}Let $G=GL(n,\mathbb{C})$. Then, $\mathcal{M}_{r}G$
is round and its mixed Hodge polynomial is given by 
\begin{eqnarray*}
\mu_{\mathcal{M}_{r}G}\left(t,u,v\right) & = & \frac{1}{n!}\sum_{\sigma\in S_{n}}\left[\det\left(I_{n}+tuv\,M_{\sigma}\right)\right]^{r}.
\end{eqnarray*}
\end{prop}

\begin{proof}
This formula is a direct application of Proposition \ref{thm:mHs-General},
since the maximal torus is $T\cong(\mathbb{C}^{*})^{n}$, $W\cong S_{n}$
acts by permutation, and $|S_{n}|=n!$. So, the automorphism $A_{\sigma}$
on $H^{1}(T,\mathbb{C})\cong\mathbb{C}^{n}$ acts by permutation,
given by the matrix $M_{\sigma}$. 
\end{proof}
\begin{rem}
There is a strong deformation retract from $\mathcal{M}_{r}G\cong\sym^{n}(\mathbb{C}^{*})^{r}$
to $(S^{1})^{r}/S_{n}\cong\sym^{n}(S^{1})^{r}$ (see \cite{FL}) which
is the space of $n$ (unordered) points on the compact $r$-torus
$(S^{1})^{r}$. So our results relate also to the study of cohomology
of so-called configuration spaces on compact Lie groups. 
\end{rem}

We now provide an even more concrete formula, and better adapted to
computer calculations, using the relation between conjugacy classes
of permutations and partitions of a natural number $n$, to compute
the above determinants.

For this, we set up some notation. Let $n\in\mathbb{N}$ and $\mathcal{P}_{n}$
be the set of partitions of $n$. We denote by $\underline{n}$ a
general partition in $\mathcal{P}_{n}$ and write it as 
\[
\underline{n}=[1^{a_{1}}2^{a_{2}}\cdots n^{a_{n}}]
\]
where $a_{j}\geq0$ denotes the number of parts of $\underline{n}$
of size $j=1,\cdots,n$; then, of course $n=\sum_{j=1}^{n}j\,a_{j}$. 
\begin{thm}
\label{thm:mu-GLn}Let $G=GL(n,\mathbb{C})$, $x=uv$ and $\underline{n}=[1^{a_{1}}2^{a_{2}}\cdots n^{a_{n}}]\in\mathcal{P}_{n}$.
The mixed Hodge polynomial of $\mathcal{M}_{r}G$ is given by
\begin{eqnarray*}
\mu_{\mathcal{M}_{r}G}\left(t,x\right) & = & \sum_{\underline{n}\in\mathcal{P}_{n}}\prod_{j=1}^{n}\frac{\left(1-\left(-tx\right)^{j}\right)^{a_{j}r}}{a_{j}!\,j^{a_{j}}}.
\end{eqnarray*}
\end{thm}

\begin{proof}
To compute the determinant in Proposition \ref{prop:mHs-GLn-det},
recall that any permutation $\sigma\in S_{n}$ can be written as a
product of disjoint cycles (including cycles of length 1), whose lengths
provide a partition of $n$, say $\underline{n}(\sigma)=[1^{a_{1}}2^{a_{2}}\cdots n^{a_{n}}]$.
Moreover, any two permutations are conjugated if and only if they
give rise to the same partition, so the conjugation class of $\sigma$
uniquely determines the non-negative integers $a_{1},\cdots,a_{n}$.
If $\sigma$ is a full cycle $\sigma=(1\cdots n)\in S_{n}$, and $M_{\sigma}$
a corresponding matrix, by computing in a standard basis, we easily
obtain the conjugation invariant expression $\det(I_{n}-\lambda M_{\sigma})=1-\lambda^{n}$.
So, for a general permutation $\sigma\in S_{n}$ with cycles given
by the partition $\underline{n}(\sigma)$ we have 
\[
\det(I_{n}-\lambda M_{\sigma})=\prod_{j=1}^{n}(1-\lambda^{j})^{a_{j}}.
\]
Now, let $c_{\underline{n}(\sigma)}$ be the size of the conjugacy
class of the permutation $\sigma$, as a subset of $S_{n}$. Then
the formula of Proposition \ref{prop:mHs-GLn-det}, with $\lambda=-tuv=-tx$,
becomes: 
\begin{eqnarray*}
\mu_{\mathcal{M}_{r}G}\left(t,x\right) & = & \frac{1}{n!}\sum_{\sigma\in S_{n}}\left[\det\left(I_{n}-(-tx)\,M_{\sigma}\right)\right]^{r}\\
 & = & \frac{1}{n!}\sum_{\underline{n}\in\mathcal{P}_{n}}c_{\underline{n}}\prod_{j=1}^{n}(1-(-tx)^{j})^{a_{j}r},
\end{eqnarray*}
where we replaced the sum over permutations by the sum over partitions
$\underline{n}$ (each repeated $c_{\underline{n}}$ times). The result
then follows from the well known formula $\frac{c_{\underline{n}}}{n!}=\prod_{j=1}^{n}\frac{1}{a_{j}!\,j^{a_{j}}}$.
\end{proof}
\begin{rem}
Since $\mathcal{M}_{r}G$ is an orbifold of dimension $nr$, it satisfies
the Poincaré Duality for mixed Hodge structures (\ref{cor:quofrmeq}(1)),
and we compute: 
\begin{eqnarray}
\mu_{\mathcal{M}_{r}G}^{c}\left(t,x\right) & = & \left(-t\right)^{nr}\sum_{\underline{n}\in\mathcal{P}_{n}}\prod_{j=1}^{n}\frac{\left((-tx)^{j}-1\right)^{a_{j}r}}{a_{j}!\,j^{a_{j}}}.\label{eq:mu-GLn}
\end{eqnarray}
\end{rem}

We now obtain the mixed Hodge polynomial for $\mathcal{M}_{r}SL(n,\mathbb{C})$,
by relating it to $\mathcal{M}_{r}GL(n,\mathbb{C})$. 
\begin{thm}
\label{thm:sl_n.gl_n}The mixed Hodge polynomials of the free abelian
character varieties of $GL(n,\mathbb{C})$ and $SL(n,\mathbb{C})$
are related by $\mu_{\mathcal{M}_{r}GL(n,\mathbb{C})}\left(t,x\right)=\left(1+tx\right)^{r}\mu_{\mathcal{M}_{r}SL(n,\mathbb{C})}\left(t,x\right)$,
giving:\textbf{
\begin{eqnarray*}
\mu_{\mathcal{M}_{r}SL(n,\mathbb{C})}\left(t,x\right) & = & \sum_{\underline{n}\in\mathcal{P}_{n}}\frac{1}{\left(1+tx\right)^{r}}\prod_{j=1}^{n}\frac{\left(1-(-tx)^{j}\right)^{a_{j}r}}{a_{j}!\,j^{a_{j}}}.
\end{eqnarray*}
}
\end{thm}

\begin{proof}
In contrast to the projection $GL(n,\mathbb{C})\to PGL(n,\mathbb{C}):=GL(n,\mathbb{C})/\mathbb{C}^{*}$
there is no algebraic map $GL(n,\mathbb{C})\to SL(n,\mathbb{C})$
that commutes with the Weyl group action. So, we need to resort to
the equivariant framework. We will actually prove a stronger equality:
\begin{eqnarray*}
\mu_{T^{r}}^{S_{n}}\left(t,x\right) & = & \left(1+tx\right)^{r}\mu_{\widetilde{T}^{r}}^{S_{n}}\left(t,x\right),
\end{eqnarray*}
where $T\cong\left(\mathbb{C}^{*}\right)^{n}$ and $\widetilde{T}=\left\{ z\in T\,|\,z_{1}\cdots z_{n}=1\right\} $
are the maximal torus of $GL(n,\mathbb{C})$, and of $SL(n,\mathbb{C})$,
respectively, and the $S_{n}$ action is the natural permutaion action
on the coordinates. From Corollary \ref{cor:quofrmeq}(2), Proposition
\ref{Prop:irred-comp}, and from the irreducibility of these character
varieties, the theorem will follow. Using the multiplicativity of
the equivariant polynomials (Proposition \ref{prop:equMhstructures}(2))
it suffices to show this for $r=1$. Consider the fibration of quasi-projective
varieties 
\[
\widetilde{T}\longrightarrow T\stackrel{\pi}{\longrightarrow}\mathbb{C}^{*},
\]
where $\pi\left(z_{1},\ldots,z_{n}\right)=z_{1}\cdots z_{n}$. By
considering the trivial action on $\mathbb{C}^{*}$, this is a fibration
of $S_{n}$-varieties with trivial monodromy, since it is in fact
a $\widetilde{T}$-principal bundle (and $\widetilde{T}$ is a connected
Lie group). Then Theorem \ref{thm:multpl} gives us an equality of
the equivariant $E$-polynomials: 
\[
E^{S_{n}}(T)=E^{S_{n}}(\tilde{T})\,E(\mathbb{C}^{*})=E^{S_{n}}(\tilde{T})(1-x).
\]
Finally, the desired formula comes from the relations in Proposition
\ref{prop:round-necessary}, since all varieties in consideration
are round. 
\end{proof}
Now, we turn to the computation of some $E^{c}$-polynomials, which
relate to some formulas obtained in \cite{LM}. 
\begin{cor}
\label{cor:E-poly-GLn}For $r\in\mathbb{N}$ and $G=GL(n,\mathbb{C})$
we have 
\[
E_{\mathcal{M}_{r}G}^{c}(x)=\sum_{\underline{n}\in\mathcal{P}_{n}}\prod_{j=1}^{n}\frac{(x^{j}-1)^{a_{j}r}}{a_{j}!\,j^{a_{j}}},
\]
with $\underline{n}=[1^{a_{1}}\cdots n^{a_{n}}]\in\mathcal{P}_{n}$,
and $\chi\left(\mathcal{M}_{r}G\right)=0.$
\end{cor}

\begin{proof}
The formula for $E_{\mathcal{M}_{r}G}^{c}(x)$ follows directly from
Equation \eqref{eq:mu-GLn} with $t=-1$. Given the previous Theorem,
the vanishing of the Euler characteristic is clear, since all factors
$x^{j}-1$ in $E_{\mathcal{M}_{r}G}^{c}$ vanish when $x=1$, and
$\chi(\mathcal{M}_{r}G)=E_{\mathcal{M}_{r}G}^{c}(1)$ by Remark \ref{rem:duality}(1). 
\end{proof}
We now prove Theorem \ref{thm:E-SLn}. 
\begin{thm}
Let $G=SL(n,\mathbb{C})$. Then, we have $\chi(\mathcal{M}_{r}G)=n^{r-1}$
and 
\[
E_{\mathcal{M}_{r}G}^{c}(x)=\sum_{\underline{n}\in\mathcal{P}_{n}}\frac{1}{\delta(\underline{n})}\,p_{\underline{n}}(x)^{r},
\]
where $\delta(\underline{n}):=\prod_{j=1}^{n}a_{j}!\,j^{a_{j}}$,
and $p_{\underline{n}}(x)=\frac{1}{x-1}\prod_{j=1}^{n}(x{}^{j}-1)^{a_{j}}$. 
\end{thm}

\begin{proof}
The formula for $E_{\mathcal{M}_{r}G}^{c}(x)$ follows immediately
from Theorem \ref{thm:sl_n.gl_n} and Corollary \ref{cor:E-poly-GLn}.
For the Euler characteristic, we need to compute $E_{\mathcal{M}_{r}G}^{c}(1)$.
First note that $p_{\underline{n}}(x)$ is a polynomial and can be
factorized as $p_{\underline{n}}(x)=(x-1)^{m}h(x)$ with $h\in\mathbb{Z}[x]$
and $m=\left(\sum_{j=1}^{n}a_{j}\right)-1$. So, $p_{\underline{n}}(1)=0$
unless $\sum_{j=1}^{n}a_{j}=1$. The only partition with $\sum_{j=1}^{n}a_{j}=1$
is $\underline{n}=[n^{1}]$ (just one part, of length $n$), which
corresponds to a cyclic permutation such as $(1\cdots n)\in S_{n}$.
The size of its conjugacy class is $c_{\underline{n}}=(n-1)!$ and
we get 
\[
\chi\left(\mathcal{M}_{r}G\right)=E_{\mathcal{M}_{r}G}^{c}(1)=\frac{1}{n!}(n-1)!\lim_{x\to1}\left(\frac{x^{n}-1}{x-1}\right)^{r}=\frac{1}{n}\ n^{r}=n^{r-1},
\]
as we wanted to show. 
\end{proof}
\begin{rem}
The $E^{c}$-polynomials of $\mathcal{M}_{r}G$ for $SL(2,\mathbb{C})$
and $SL(3,\mathbb{C})$ are already present in \cite{LM}. For $n\geq4$
these formulas are new, and can also be upgraded to mixed Hodge polynomials
by using Remark \ref{rem:duality}(1) and Poincaré duality. 
\end{rem}

\begin{example}
\label{exa:table-SLn}The following table gives the explicit values
of $\delta(\underline{n})$ and $p_{\underline{n}}(x)$ up to $n=5$
(in each row, the ordering is preserved). All the formulas can be
easily implemented in the available computer software packages (in
this paper, most of our calculations were performed with GAP). For
simplicity, the notation $[12]$ refers to a partition of $n=3$ with
two cycles: one of length 1, another of length 2 (not a cycle of length
twelve).\\[-3mm]

\noindent %
\begin{tabular}{c|cccc}
\multicolumn{1}{c|}{$n$} & $|\mathcal{P}_{n}|$  & $\underline{n}$  & $\delta(\underline{n})$  & $p_{\underline{n}}(x)$\tabularnewline
\hline 
\hline 
$2$  & $2$  & $[2];\ [1^{2}]$  & $2;\ 2$  & $x+1;\ x-1$\tabularnewline
\hline 
$3$  & $3$  & $\begin{array}{c}
[3];\\{}
[12];\ [1^{3}]
\end{array}$  & $\begin{array}{c}
3;\\
2;\ 6
\end{array}$  & $\begin{array}{c}
x^{2}+x+1;\\
x^{2}-1;\ \left(x-1\right)^{2}
\end{array}$\tabularnewline
\hline 
$4$  & $5$  & $\begin{array}{c}
[4];\\{}
[13];\ [2^{2}];\\{}
[1^{2}2];\ [1^{4}]
\end{array}$  & $\begin{array}{c}
4;\\
3;\ 8;\\
4;\,24
\end{array}$  & $\begin{array}{c}
x^{3}+x^{2}+x+1;\\
x^{3}-1;\ \left(x^{2}-1\right)\left(x+1\right);\\
\left(x-1\right)^{2}\left(x+1\right);\ \left(x-1\right)^{3}
\end{array}$\tabularnewline
\hline 
$5$  & $7$  & $\begin{array}{c}
[5];\ [14];\\{}
[1^{2}3];\ [23];\\{}
[12^{2}];\ [1^{3}2];\ [1^{5}]
\end{array}$  & $\begin{array}{c}
5;\ 4;\\
6;\ 6;\\
8;\ 12;\ 120
\end{array}$  & $\begin{array}{c}
\frac{x^{5}-1}{x-1};\ x^{4}-1;\\
\left(x^{3}-1\right)\left(x-1\right);\ \left(x^{3}-1\right)\left(x+1\right);\\
\left(x^{2}-1\right)^{2};\ \left(x-1\right)^{3}(x+1);\ \left(x-1\right)^{4}
\end{array}$\tabularnewline
\hline 
\end{tabular}
\end{example}

For example, with $n=4$, the table immediately gives 
\[
E^{c}(x)=\frac{1}{4}(x^{3}+x^{2}+x+1)^{r}+\frac{1}{3}(x^{3}-1)^{r}+\frac{1}{8}(x^{2}-1)^{r}(x+1)^{r}+\frac{1}{4}(x-1)^{2r}(x+1)^{r}+\frac{1}{24}(x-1)^{3r},
\]
for any $r\geq1$. 

\subsection{Symmetric products and Cheah's formula. }

So far, our approach to Hodge numbers for the character varieties
$\mathcal{M}_{r}G$, for $G=GL(n,\mathbb{C})$ and $SL(n,\mathbb{C})$
is well adapted to fixing $n\in\mathbb{N}$, and let $r$ be arbitrary,
as we can see from theorems \ref{thm:mu-GLn} and \ref{cor:E-poly-GLn}.
On the other hand, since the $GL(n,\mathbb{C})$-character varieties
of $\mathbb{Z}^{r}$ are symmetric products 
\[
\mathcal{M}_{r}G=\mathcal{M}_{r}GL(n,\mathbb{C})=\sym^{n}(\mathbb{C}^{*})^{r},
\]
as in Proposition \ref{prop:GLn-char-var}, we can apply a formula
of J. Cheah \cite{Ch} for the mixed Hodge numbers of symmetric products.
Indeed, this will lead to an ``orthogonal'' approach: by fixing
small values for $r\in\mathbb{N}$, we obtain simple formulas valid
for all $n\in\mathbb{N}$.

Let $X$ be a quasi-projective variety with given compactly supported
Hodge numbers $h_{c}^{k,p,q}(X)$. Cheah's formula gives the generating
function of the mixed Hodge polynomials of all symmetric products
$X^{(n)}=\sym^{n}X$ is (see \cite{Ch}): 
\begin{equation}
\sum_{n\geq0}\,\mu_{X^{(n)}}^{c}(t,u,v)\,z^{n}=\prod_{p,q,k}\left(1-(-1)^{k}u^{p}v^{q}t{}^{k}z\right)^{(-1)^{k+1}h_{c}^{k,p,q}(X)}.\label{eq:Cheah}
\end{equation}

We start by observing that, for varieties satisfying Poincaré duality,
Cheah's formula stays unaffected when passing from $\mu^{c}$ to $\mu$
and from $h_{c}^{k,p,q}$ to $h^{k,p,q}$. 
\begin{prop}
\label{prop:Cheah-usual-H}Let $X$ satisfy Poincaré duality. Then
\[
\sum_{n\geq0}\,\mu_{X^{(n)}}(t,u,v)\,z^{n}=\prod_{p,q,k}\left(1-(-1)^{k}u^{p}v^{q}t{}^{k}z\right)^{(-1)^{k+1}h^{k,p,q}(X)}.
\]
\end{prop}

\begin{proof}
This is a simple calculation. Let $X$ have complex dimension $d$.
From the relation between $\mu_{X}$ and $\mu_{X}^{c}$, in Remark
\ref{rem:duality}(1), Cheah's formula \eqref{eq:Cheah} is equivalent
to: 
\begin{eqnarray*}
\sum_{n\geq0}\,\mu_{X^{(n)}}(t^{-1},u^{-1},v^{-1})\,(t^{2}uv)^{nd}z^{n} & = & \prod_{p,q,k}\left(1-(-1)^{k}u^{p}v^{q}t{}^{k}z\right)^{(-1)^{k+1}h^{2d-k,d-p,d-q}(X)}
\end{eqnarray*}
Now, changing the indices $(k,p,q)$ to $(k',p',q')=(2d-k,d-p,d-q)$,
which preserves the parity of $k$, we obtain 
\[
\sum_{n\geq0}\,\mu_{X^{(n)}}\left(\frac{1}{t},\frac{1}{u},\frac{1}{v}\right)\left((t^{2}uv)^{d}z\right)^{n}=\prod_{p',q',k'}\left(1-(-1)^{k'}u^{-p'}v^{-q'}t{}^{-k'}(t^{2d}u^{d}v^{d}z)\right)^{(-1)^{k'+1}h^{k',p',q'}(X)}
\]
which is clearly equivalent to the desired formula, under the substitution:\\
 $(t^{-1},u^{-1},v^{-1},t^{2d}u^{d}v^{d}z)\mapsto(t,u,v,z)$. 
\end{proof}
For round varieties, as before, all products reduce to a single index. 
\begin{prop}
\label{prop:Cheah-round}Let $X$ be a round variety of dimension
$d$ satisfying Poincaré duality. Then 
\[
\sum_{n\geq0}\,\mu_{X^{(n)}}(t,u,v)\,z^{n}=\prod_{k}\left(1-(-tuv)^{k}z\right)^{(-1)^{k+1}h^{k,k,k}(X)}.
\]
\end{prop}

\begin{proof}
This is immediate from Proposition \eqref{prop:Cheah-usual-H}, since
the only non-zero Hodge numbers of a round variety $X$ are $h^{k,k,k}(X)$,
for some $k$. 
\end{proof}
We are now ready to apply this formula to $\mathcal{M}_{r}GL(n,\mathbb{C})$.
Since this space is $\sym^{n}(\mathbb{C}^{*})^{r}$, we should consider
$X=(\mathbb{C}^{*})^{r}$. 
\begin{cor}
\label{cor:Cheah-for-M_r}Let $G=GL(n,\mathbb{C})$. Then: 
\[
\sum_{n\geq0}\,\mu_{\mathcal{M}_{r}G}(t,u,v)\,z^{n}=\prod_{k\geq0}\left(1-(-tuv)^{k}z\right)^{(-1)^{k+1}\binom{r}{k}}=\frac{\prod_{k\,\odd}\left(1+(tuv)^{k}z\right)^{\binom{r}{k}}}{\prod_{k\,\even}\left(1-(tuv)^{k}z\right)^{\binom{r}{k}}}.
\]
\end{cor}

\begin{proof}
Letting $X=(\mathbb{C}^{*})^{r}$, $d=r=\dim X$, and $h^{k,k,k}(X)=\binom{r}{k}$,
$0\leq k\leq r$, the proof is immediate from Proposition \eqref{prop:Cheah-round}.
\end{proof}
\begin{example}
The simplest example of this formula is when $r=1$ {[}For $r=0$,
$\sym^{n}(\mathbb{C}^{*})^{0}$ is a single point{]}. In this case,
$X=\mathbb{C}^{*}$ and $\binom{1}{0}=\binom{1}{1}=1$, so we expand
the right hand-side as a power series in $z$, as: 
\[
\frac{1+tuvz}{1-z}=1+\sum_{n\geq1}(1+uvt)\,z^{n}
\]
In particular, $\mu_{\mathcal{M}_{1}GL(n,\mathbb{C})}(t,u,v)=1+tuv$,
for $n\geq1$, which agrees with the fact that $\mathcal{M}_{1}GL(n,\mathbb{C})=\sym^{n}(\mathbb{C}^{*})\cong\mathbb{C}^{n-1}\times\mathbb{C}^{*}$
has the same Hodge structure than $\mathbb{C}^{*}$. 
\end{example}

The case with $r=2$ is an interesting result in itself. 
\begin{prop}
\label{prop.r=00003D2}Let $G=GL(n,\mathbb{C})$, $r=2$ and $n\geq1$.
Then: 
\[
\mu_{\mathcal{M}_{2}G}(t,u,v)=(1+tuv)^{2}(1+(tuv)^{2}+...+(tuv)^{2n-2}),
\]
$P_{\mathcal{M}_{2}G}(t)=(1+t)^{2}(1+t^{2}+\cdots+t^{2n-2})$ and
$\chi(\mathcal{M}_{2}G)=0$.
\end{prop}

\begin{proof}
We now have $\binom{2}{2}=\binom{2}{0}=1$ and $\binom{2}{1}=2$ so
we expand the right hand-side of Corollary \ref{cor:Cheah-for-M_r}
as a power series in $z$, writing $\lambda=tuv$ for simplicity:
\begin{eqnarray*}
\frac{\left(1+\lambda z\right)^{2}}{(1-z)(1-\lambda^{2}z)} & = & 1+\frac{(1+\lambda)^{2}}{1-\lambda^{2}}\left(\frac{1}{1-z}-\frac{1}{1-\lambda^{2}z}\right)\\
 & = & 1+(1+\lambda)^{2}\sum_{n\geq0}\left(1+\lambda^{2}+\cdots+\lambda^{2n-2}\right)z{}^{n}.
\end{eqnarray*}
This gives the desired formulas for $\mathcal{M}_{2}GL(n,\mathbb{C})$,
with $\lambda=tuv$, with $uv=1$ for $P_{\mathcal{M}_{2}G}$ and
$\lambda=-1$ for the Euler characteristic. 
\end{proof}
The next corollary follows immediately from Theorem \ref{thm:sl_n.gl_n}. 
\begin{cor}
Let $G=SL(n,\mathbb{C})$, $r=2$ and $n\geq1$. Then: 
\[
\mu_{\mathcal{M}_{2}G}(t,x)=1+(tx)^{2}+...+(tx)^{2n-2},
\]
$P_{\mathcal{M}_{2}G}(t)=(1+t^{2}+\cdots+t^{2n-2})$ and $\chi(\mathcal{M}_{2}G)=n$.
\end{cor}

\begin{rem}
The equality of Poincaré polynomials $P_{\mathcal{M}_{2}SL(n,\mathbb{C})}=P_{\mathbb{P}_{\mathbb{C}}^{n-1}}$
is not a coincidence. In fact, by non-abelian Hodge correspondence,
$\mathcal{M}_{2}SL(n,\mathbb{C})$ is diffeomorphic to the cotangent
bundle of the projective space $\mathbb{P}_{\mathbb{C}}^{n-1}$ parametrizing
semistable bundles over an elliptic curve of rank $n$ and trivial
determinant (see \cite{BF,Tu}). 
\end{rem}

\subsection{A combinatorial identity}

We finish the article with an interesting purely combinatorial identity.
We could not find out whether this identity was noticed before. Recall
that $I_{n}$ is the identity $n\times n$ matrix and $M_{\sigma}$
a permutation matrix associated to $\sigma\in S_{n}$. 
\begin{thm}
Fix $r\in\mathbb{N}_{0}$. Then, for formal variables $x,z$ (or considering
$x,z\in\mathbb{C}$ in a small disc around the origin) we have: 
\[
\prod_{k\geq0}\left(1-x^{k}z\right)^{(-1)^{k+1}\binom{r}{k}}=\sum_{n\geq0}\sum_{\sigma\in S_{n}}\frac{z^{n}}{n!}\det(I_{n}-xM_{\sigma})^{r}.
\]
\end{thm}

\begin{proof}
Putting together Corollary \ref{cor:Cheah-for-M_r} and the formula
for $\mu_{\mathcal{M}_{r}GL(n,\mathbb{C})}$ in Proposition \ref{prop:mHs-GLn-det}
we obtain 
\begin{eqnarray*}
\sum_{n\geq0}\,\frac{1}{n!}\sum_{\sigma\in S_{n}}\left[\det\left(I_{n}+tuv\,M_{\sigma}\right)\right]^{r}\,z^{n} & = & \prod_{k\geq0}\left(1-(-tuv)^{k}z\right)^{(-1)^{k+1}\binom{r}{k}}
\end{eqnarray*}
which becomes the desired identity, by setting $x=-tuv$. Note that
for $r=0$ the formula is still valid and reduces to the geometric
series. The formula holds also for $z,x\in\mathbb{C}$ where the series
converges. We readily check that convergence holds whenever $|x|<1$
and $|z|<2^{-r}$, using the bound, valid for $|x|<1$, 
\[
\sum_{\sigma\in S_{n}}\prod_{j=1}^{n}(1-x^{j})^{a_{j}r}<\sum_{\sigma\in S_{n}}2^{|\sigma|r}<n!\,2^{nr},
\]
where $|\sigma|=\sum_{j}a_{j}$ denotes the number of cycles of $\sigma$,
with $a_{j}$ parts of size $j$, as in the proof of Theorem \ref{thm:mu-GLn}. 
\end{proof}

\appendix
%dummy comment 

\section{Multiplicativity of the $E$-polynomial under Fibrations}

In this appendix, we prove a multiplicative property of the $E$-polynomial
under fibrations, used in Theorem \ref{thm:sl_n.gl_n}. This is a
consequence of the fact that the Leray-Serre spectral sequence is
a spectral sequence of mixed Hodge structures.

\subsection{$E$-polynomials of fibrations}

It is well known that for an algebraic fibration of algebraic varieties
\[
Z\longrightarrow E\stackrel{\pi}{\longrightarrow}B
\]
the Poincaré polynomials do not behave multiplicatively, in general.
Then, a fortiori, a multiplicative property is not expected for the
Hodge-Deligne polyomial. But when all involved varieties are smooth
and the associated monodromy is trivial, this property turns out to
be valid for the $E$-polynomials. Moreover, if there is a finite
group $F$ acting on the three varieties, and the involved maps respect
the action - an equivariant fibration of $F$-varieties -, one gets
a multiplicative formula for their equivariant polynomials, under
the assumption that the higher direct images sheaves $R^{j}\pi_{*}\mathbb{C}_{E}$
(associated to the presheaf $U\mapsto H^{j}\left(\pi^{-1}U,\mathbb{C}_{E}\right)$
for $U\subset B$) are constant, where $\mathbb{C}_{E}$ is the constant
sheaf on $E$. 
\begin{thm}
\label{thm:multpl}Let $F$ be a finite group and consider an algebraic
fibration between smooth complex algebraic quasi-projective varieties
\[
\xymatrix{Z\ar[r] & E\ar[r]^{\pi} & B}
\]
(not necessarily locally trivial in the Zariski topology). Suppose
also that this is a fibration of $F$-varieties (all spaces are $F$-varieties
and the maps are $F$-equivariant). If $Z$ is connected and $R^{j}\pi_{*}\mathbb{C}_{E}$
are constant for every $j$, then 
\begin{eqnarray*}
E_{E}^{F}\left(u,v\right) & = & E_{Z}^{F}\left(u,v\right)\varotimes E_{B}^{F}\left(u,v\right).
\end{eqnarray*}
\end{thm}

\begin{proof}
The non-equivariant version of this result is the content of Proposition
2.4 in \cite{LMN}, where it is used to calculate the Serre polynomials
of certain twisted character varieties. We detail the argument here,
for the reader's convenience. Firstly, assume that the $F$-action
is trivial on the three spaces. The Leray-Serre spectral sequence
of the fibration is a sequence of mixed Hodge structures (\cite[Theorem 6.5]{PS})
and it is proved in \cite[Theorem 6.1]{DL} that under the given assumptions,
its second page $E_{2}^{a,b}$ admits an isomorphism 
\begin{eqnarray*}
E_{2}^{a,b} & \simeq & H^{a}\left(B\right)\varotimes H^{b}\left(Z\right),
\end{eqnarray*}
which is actually an isomorphism of mixed Hodge structures. In particular,
we get an equality between their respective graded pieces: 
\begin{eqnarray}
Gr_{F}^{p}Gr_{p+q}^{W}E_{2}^{a,b} & = & \bigoplus_{p'+p''=p}\,\bigoplus_{q'+q''=q}Gr_{F}^{p'}Gr_{p'+q'}^{W}H^{a}\left(B\right)\varotimes Gr_{F}^{p''}Gr_{p''+q''}^{W}H^{b}\left(Z\right)\nonumber \\
 & = & \bigoplus_{p'+p''=p}\,\bigoplus_{q'+q''=q}H^{a,p',q'}(B)\varotimes H^{b,p'',q''}(Z)\label{eq:S.s.m.H.s}
\end{eqnarray}
Using once more the fact that this is a spectral sequence of mixed
Hodge structures, we get one spectral sequence of vector spaces for
each pair $\left(p,q\right)$: 
\[
\begin{array}{ccccc}
E(p,q)_{2}^{a,b} & := & Gr_{F}^{p}Gr_{p+q}^{W}E_{2}^{a,b} & \Rightarrow & Gr_{F}^{p}Gr_{p+q}^{W}H^{a+b}\left(E\right)\end{array}.
\]
Now set 
\begin{eqnarray*}
P_{\left(p,q\right)}\left(t\right) & := & \sum_{k\geq0}\dim\left(\bigoplus_{a+b=k}E\left(p,q\right)_{2}^{a,b}\right)t^{k}.
\end{eqnarray*}
Given that $\oplus_{a+b=k}E(p,q)_{2}^{a,b}\Rightarrow\oplus_{a+b=k}Gr_{F}^{p}Gr_{p+q}^{W}H^{a+b}\left(E\right)$,
one has 
\begin{eqnarray*}
P_{\left(p,q\right)}\left(-1\right) & = & \sum_{k}\left(-1\right)^{k}\dim Gr_{F}^{p}Gr_{p+q}^{W}H^{k}\left(E\right)\\
 & = & \sum_{k}\left(-1\right)^{k}h^{k,p,q}\left(E\right).
\end{eqnarray*}
So, by definition, $E_{E}\left(u,v\right)=\sum_{p,q}P_{\left(p,q\right)}\left(-1\right)u^{p}v^{q}$.
On the other hand, using \eqref{eq:S.s.m.H.s} 
\[
P_{\left(p,q\right)}\left(-1\right)=\sum_{k}\sum_{a+b=k}\sum_{p'+p''=p}\sum_{q'+q''=q}\left(-1\right)^{a}h^{a,p',q'}\left(B\right)\left(-1\right)^{b}h^{b,p'',q''}\left(Z\right),
\]
substituting into $E_{E}\left(u,v\right)=\sum_{p,q}P_{\left(p,q\right)}\left(-1\right)u^{p}v^{q}$,
and switching summation order, one gets $E_{E}\left(u,v\right)=E_{B}\left(u,v\right)E_{Z}\left(u,v\right)$,
as wanted. Succinctly, the argument follows from the fact that the
spectral sequence above can be seen as an equality between derived
functors and, thinking in terms of K-theory, passing to cohomology
for obtaining the next sheet in the sequence does not change an alternating
sum. 

Finally, to prove the equivariant version, suppose that all the cohomologies
are $F$-modules. Then, since we have a fibration of $F$-varieties
the associated Leray-Serre spectral sequence is a spectral sequence
of $F$-modules. The associated sequences $E\left(k,m\right)_{2}^{p,q}$
are also spectral sequences of $F$-modules, since the graded pieces
for the Hodge and weight filtration are so. To get the desired equality,
it suffices to proceed as before: in each step of the we substitute
the dimension $h^{k,p,q}(\cdot)$ by the corresponding $F$-module
$\left[Gr_{F}^{p}Gr_{p+q}^{W}H^{k}\left(\cdot\right)\right]_{F}$,
and the operations performed in $\left(\mathbb{Z},+,\times\right)$
are replaced by those in the ring $\left(R\left(F\right)\left[u,v\right],\oplus,\varotimes\right)$.
\end{proof}
\begin{rem}
The above proof follows the proof of Theorem 6.1 (ii) in \cite{DL},
where the version for the equivariant weight polynomial is obtained
(which is implied by Theorem \ref{thm:multpl}, since the weight polynomial
is a specialization of the $E$-polynomial). We also remark that the
study of multiplicative invariants under fibrations, goes back at
least to work of Chern, Hirzebruch and Serre in the mid fifties \cite{CHS},
on the signature theorem.
\end{rem}

\bibliographystyle{alpha}

\end{document}